\documentclass[11pt,letterpaper]{article}
\usepackage{amsfonts, amsmath, amssymb, amscd, amsthm, color, graphicx, mathabx, mathrsfs, wasysym, setspace, mdwlist, calc, float, mathtools, tikz-cd, MnSymbol}
\usepackage{setspace}

\usepackage{enumitem}
\usepackage{tocloft}
\usepackage{xcolor}

\usepackage{hyperref}

\hypersetup{colorlinks,
    linkcolor={red!50!black},
    citecolor={blue!80!black},
    urlcolor={blue!80!black}}
\usepackage{float}

\numberwithin{equation}{section}

\renewcommand{\phi}{\varphi}
\newcommand{\WR}{\mathcal{WR}}
\renewcommand{\wr }{\,{\rm wr}\, }

\newcommand{\C }{\mathcal C}

\newcommand{\e }{\varepsilon }

\renewcommand{\P }{\mathcal P}
\renewcommand{\kappa }{\varkappa}

\renewcommand{\ll }{\langle\hspace{-.7mm}\langle }
\newcommand{\rr }{\rangle\hspace{-.7mm}\rangle }

\newcommand{\NN}{\mathbb N}

\newcommand{\M}{\mathcal M}
\newcommand{\Nn}{\mathcal N}
\newcommand{\ra}{\rightarrow}
\newcommand{\ca}{\curvearrowright}
\newcommand{\A}{\mathcal A}

\newcommand{\D}{\mathcal D}

\newcommand{\Q}{\mathcal Q}
\newcommand{\R}{\mathcal R}
\newcommand{\T}{\mathcal T}

\newcommand{\W}{\mathcal W}

\newcommand{\sU}{\mathscr U}
\newcommand{\sN}{\mathscr N}

\newcommand{\sZ}{\mathscr Z}
\newcommand{\sR}{\mathscr R}

\newcommand{\ZZ}{\mathbb Z}

\newtheorem{thm}{Theorem}[section]
\newtheorem*{thm*}{Theorem}
\newtheorem{cor}[thm]{Corollary}
\newtheorem{lem}[thm]{Lemma}
\newtheorem{claim}[thm]{Claim}

\theoremstyle{definition}
\newtheorem{defn}[thm]{Definition}

\newtheorem{ex}[thm]{Example}

\theoremstyle{remark}
\newtheorem{rem}[thm]{Remark}

\begin{document}

\title{Wreath-like products of groups and their von Neumann algebras III: Embeddings}

\author{Ionu\c t Chifan, Adrian Ioana, Denis Osin and Bin Sun}

\date{}

\maketitle

\begin{abstract}
For a class of wreath-like product groups with property (T), we describe explicitly all the embeddings between their von Neumann algebras. This allows us to provide a continuum of ICC groups with property (T) whose von Neumann algebras are pairwise non (stably) embeddable.
We also give a construction of groups in this class only having inner injective homomorphisms. As an application, we obtain examples of group von Neumann algebras which admit only inner endomorphisms.
\end{abstract}


\section{Introduction}
This paper is a continuation of our works 
\cite{CIOS1,CIOS3} where we introduced a new class of groups, called wreath-like products, and proved several 
rigidity results for their von Neumann algebras. 
A group $G$ is called a {\it wreath-like product} of two groups $A$ and $B$ corresponding to an action $B\curvearrowright I$ if $G$ is an extension of the form 
$$1\longrightarrow \bigoplus_{i\in I}A_i \longrightarrow  G \stackrel{\e}\longrightarrow B\longrightarrow 1,$$
where $A_i\cong A$ and the action of $G$ on $\bigoplus_{i\in I}A_i$ by conjugation satisfies 
$gA_ig^{-1} = A_{\e(g)i}$ for all $i\in I$. 
If $I=B$ and $B\curvearrowright I$ is the left regular action, $G$ is called a \emph{regular wreath-like product} of $A$ and $B$.
The set of wreath-like products of $A$ and $B$ corresponding to $B\curvearrowright I$ (respectively, regular wreath-like products) is denoted by $\WR(A, B\curvearrowright I)$ (respectively, $\WR(A,B)$).

Wreath-like products generalize ordinary wreath products of groups and exhibit similar algebraic properties. However, a notable distinction between these two classes arises in the analytical context. For our purpose, the most significant difference lies in the fact that the ordinary wreath product $A\wr B$ can have Kazhdan's property (T) only in ``degenerate" cases (when $A$ is trivial or $B$ is finite); in contrast, numerous ``non-degenerate" wreath-like products with property (T) exist. Examples of the latter kind arise naturally in the setting of group theoretic Dehn filling as shown in \cite{CIOS1}.  

The main results of \cite{CIOS1,CIOS3} show that von Neumann algebras of wreath-like products with property (T) are remarkably rigid and retain a lot of information about the group. For a countable discrete group, we denote by $\text{L}(G)$ its group von Neumann algebra \cite{MvN36}. Recall that $\text{L}(G)$ is a II$_1$ factor if and only if $G$ is ICC, i.e., all the nontrivial conjugacy classes of $G$ are infinite \cite{MvN43}. As a consequence of 
Connes' work \cite{Co76} all II$_1$ factors $\text{L}(G)$ with $G$ ICC amenable are pairwise isomorphic. 
In contrast, Connes' rigidity conjecture \cite{Co82} predicts every ICC property (T) group $G$ is W$^*$-superrigid,  i.e., if $\text{L}(G)\cong\text{L}(H)$, then $G\cong H$, for any countable discrete group $H$.

Let $\mathcal C_1$ be the class of wreath-like product groups belonging to $\W\R(A,B\curvearrowright I)$, for some nontrivial abelian group $A$, nontrivial ICC subgroup $B$ of a hyperbolic group and action $B\curvearrowright I$ with amenable stabilizers. In \cite[Theorem 1.3]{CIOS1}
we proved that any property (T) group $G\in\mathcal C_1$ is W$^*$-superrigid, thus providing examples of W$^*$-superrigid groups with property (T). 
Subsequently, in \cite{CIOS3}, we studied a wider class of wreath-like products $\mathcal C_2\supset\mathcal C_1$, defined by relaxing the assumptions on $B$ and $B\curvearrowright I$. We showed that property (T) groups $G,H$ from $\mathcal C_2$ satisfy a strong form of Connes' conjecture proposed by Popa in \cite{Po06}: 
any $*$-isomorphism $\theta\colon\text{L}(G)\cong\text{L}(H)$ is induced by a character of $G$ and a group isomorphism $\delta\colon G\cong H$, up to unitary conjugacy. In particular, this result confirmed Jones' conjecture \cite{Jo00} 
that Out$(\text{L}(G))\cong\text{Char}(G)\rtimes\text{Out}(G)$.  
This further allowed us to deduce that any countable group arises as $\text{Out}(\text{L}(G))$, for some ICC property (T) group $G$.

Our goal in this paper is to expand the scope of 
the rigidity phenomena obtained in \cite{CIOS1,CIOS3} 
from {\it isomorphisms} to {\it embeddings}. More precisely, we show that any embedding $\theta\colon\text{L}(G)\hookrightarrow\text{L}(H)$ between any property (T) groups $G$, $H$ from a large subclass of $\mathcal C_1$ is induced by a character of $G$ and an embedding $\delta\colon G\hookrightarrow H$, up to unitary conjugacy. To state our main result, we need an auxiliary definition.

\begin{defn}\label{C_0}
Let $\mathcal C_0$ be the set of all groups  $G\in \mathcal W\mathcal R(A,B\curvearrowright I)$, such that the following conditions hold.
\begin{enumerate} 
\item[(a)] $A$ is nontrivial abelian.  
\item[(b)] $B$ is a nontrivial ICC subgroup of a hyperbolic group and $\emph{C}_{ B}(g)$ is virtually cyclic for all $g\in B\setminus\{1\}$. 
\item[(c)] $\text{Stab}_B(i)$ is amenable for all $i\in I$. 
\end{enumerate}
\end{defn}

\begin{rem}
    By \cite[Lemma 4.11(a)]{CIOS1}, conditions (b) and (c) automatically imply that the action $B\curvearrowright I$ is faithful.
\end{rem}

Our main result is the following.

\begin{thm}\label{main}
Let $G$ be an ICC property (T) group admitting a normal abelian subgroup $C$ such that $\{cgc^{-1}\mid c\in C\}$ is infinite, for all $g\in G\setminus C$. Let $H\in\mathcal C_0$. 

Suppose that, for some projection $p\in\emph{L}(H)$, there exists a unital $*$-homomorphism $\theta\colon \emph{L}(G)\rightarrow p\emph{L}(H)p$. Then $p=1$ and we can find a character $\rho\colon G\rightarrow\mathbb T$, a homomorphism $\delta\colon G\rightarrow H$,  and a unitary element $u\in \emph{L}(H)$ such that $\theta(u_g)= \rho(g) u v_{\delta(g)}u^*$, for all 
 $g\in G$, where $(u_g)_{g\in G}\subset{\rm L}(G)$ and $(v_h)_{h\in H}\subset{\rm L}(H)$ denote the usual generating unitaries. 

 In particular, $\emph{L}(G)\hookrightarrow\emph{L}(H)$ if and only if $G$ embeds into $H$.
 Moreover, $\emph{L}(G)\hookrightarrow_{\emph{s}}\emph{L}(H)$ if and only if there exists a finite index subgroup $G_0<G$ which embeds into $H$.
\end{thm}

Here, following \cite{PV21}, for II$_1$ factors $\mathcal M$, $\mathcal N$, we write $\mathcal M\hookrightarrow \mathcal N$ and say that $\mathcal M$ \textit{embeds } into $\mathcal N$ if there is a unital $*$-homomorphism $\theta\colon \mathcal M\rightarrow \mathcal N$.
We write $\mathcal M\hookrightarrow_{\text{s}} \mathcal N$ and say that $\mathcal M$ \textit{stably embeds} into $\mathcal N$ if there is a unital $*$-homomorphism $\theta\colon \mathcal M\rightarrow \mathcal N^t$, for some $t>0$.

Theorem \ref{main} applies to any property (T) groups $G,H\in\mathcal C_0$.
Indeed,  if $G\in\W\R(A,B\curvearrowright I)$ belongs to $\mathcal C_0$, 
then $G$ is ICC by \cite[Lemma 4.11(b)]{CIOS1}, $\oplus_{i\in I}A<G$ is normal abelian and $\{aga^{-1}\mid a\in\otimes_{i\in I}A\}$ is infinite, for all $g\in G\setminus(\oplus_{i\in I}A)$ by \cite[Lemma 4.4.2]{CIOS3}.
To put Theorem \ref{main} into a better perspective, we note that the problem of describing all embeddings between II$_1$ factors in a given class is extremely challenging, with only a few results available in the literature \cite{Io10,De13,PV21}. 
Theorem \ref{main} makes  progress on this problem by describing all embeddings $\text{L}(G)\hookrightarrow\text{L}(H)$,
for any property (T) groups $G,H\in\mathcal C_0$.

\begin{ex}
Examples of property (T) groups belonging to $\mathcal C_0$ can be obtained in the following way.  Let $K$ be a torsion-free hyperbolic group with property (T) and let $g\in K\setminus\{1\}$. For any $k\in \mathbb N$, we define the quotient group $G=K/[\ll g^k\rr,\ll g^k\rr]$. For any sufficiently large $k\in\mathbb N$, the group $G$ has property (T) and belongs to $\mathcal C_1$ by \cite[Theorem 1.4]{CIOS1}; furthermore, Lemma \ref{Lem:cent} implies that $G\in\mathcal C_0$.
\end{ex}

The proof of Theorem \ref{main} relies on techniques from Popa's deformation/rigidity theory and is given in Section \ref{embeddings}. In fact, we establish a stronger result that applies to amplifications (Theorem \ref{rigembed}), leading to an explicit description of all embeddings $\theta:\text{L}(G)\rightarrow\text{L}(H)^t$, $t>0$ (Theorem \ref{explicit}).  The beginning of the proof of Theorem \ref{rigembed} follows closely the proof of \cite[Theorem 1.3]{CIOS1}. The main novelty lies in the rest of the proof (Claims \ref{P0}-\ref{Ji}), where we address several 
challenges that are not present in \cite{CIOS1}   but arise here as $\theta$ is not assumed surjective and so the relative commutant $\theta(\text{L}(G))'\cap \text{L}(H)^t$ may not be trivial.

We continue with some applications of Theorem \ref{main}.
In \cite[Corollary 1.7]{CIOS1} we proved the existence of a continuum 
of pairwise non-isomorphic ICC property (T) groups which are W$^*$-superrigid. Since these groups belong to the class $\mathcal C_0$, Theorem \ref{main} allows us to strengthen \cite[Corollary 1.7]{CIOS1} as follows.

\begin{cor}\label{continuum}
    There exists a continuum $\{G_i\}_{i\in I}$ of  W$^*$-superrigid, ICC property (T) groups such that $\emph{L}(G_i)\not\hookrightarrow_{\emph{s}}\emph{L}(G_{i'})$, for all $i,i'\in I$ with $i\not=i'$.
\end{cor}

Corollary \ref{continuum} provides the first infinite family of  pairwise non-embeddable II$_1$ factors arising from property (T) groups. We note that results from \cite{CH89,OP03} can be combined to provide arbitrarily large 
finite, but not infinite, such families. 
Recently, Popa and Vaes carried a systematic study of the (stable) embedding relation for separable II$_1$ factors in \cite{PV21}.
In particular, they constructed one-parameter families of mutually non-embeddable separable II$_1$ factors. However, the II$_1$ factors studied in \cite{PV21} do not have property (T).

Theorem \ref{main} also leads to novel computations of invariants of II$_1$ factors.
For a II$_1$ factor $\M$, we denote by $\text{Out}(\M)=\text{Aut}(\M)/\text{Inn}(\M)$ and $\mathcal F(\M)$ the {\it outer automorphism} and {\it fundamental} groups of $\M$.
Following \cite{Jo83},  let $\mathcal I_\mathcal M$ denote the set of all $r\in [1,\infty)\cup\{\infty\}$ such that $\mathcal M$ contains a subfactor $\mathcal N$ of index $[\mathcal M:\mathcal N]=r$. The {\it endomorphism semigroup} $\text{End}(\mathcal M)$ introduced in \cite[Definition 10.4]{Io10} is the set of all unital $*$-homomorphisms $\theta\colon \mathcal M\rightarrow \mathcal M$, and  the {\it fundamental semigroup} $\mathcal F_\text{s}(\mathcal M)$ is the set of all $t>0$ for which there is a unital $*$-homomorphism $\theta\colon \mathcal M\rightarrow \mathcal M^t$. Note that  $\{n^2\mid n\in\mathbb N\}\cup\{\infty\}\subset\mathcal I_\M\subset\mathcal F_{\text{s}}(\M)\cup\{\infty\}$, for every II$_1$ factor $M$ (see Remark \ref{index} for the second inclusion). 

If $G\in\mathcal C_0$ is a property (T) group, 
then a strengthening of Theorem \ref{main} which applies to amplifications (see Theorem \ref{rigembed}) implies that $\mathcal F_{\text{s}}(\text{L}(G))=\mathbb N$.
This recovers the fact that $\mathcal F(\text{L}(G))=\{1\}$ proved in \cite[Theorem 1.3]{CIOS1}.  Using geometric group theory toolbox and our work \cite{CIOS2}, we construct examples of such groups $G$ that admit no  nontrivial characters  and only inner injective homomorphisms. This construction is described in Section 2 (see  Theorem \ref{Thm:Group} and Lemma \ref{trivialout}), and allows to derive the following.

\begin{thm}\label{endo}
There exists an ICC property (T) group $G$ such that $\mathcal M=\emph{L}(G)$ satisfies  $\emph{Out}(\mathcal M)=\{1\}$,  $\emph{End}(\mathcal M)= \emph{Inn}(\mathcal M)$,  $\mathcal F(\mathcal M)=\{1\}$ and $\mathcal F_s(\mathcal M)=\mathbb N$. 
\end{thm}

Theorem \ref{endo} provides the first example of a group II$_1$ factor $\mathcal M$ satisfying the properties $\text{End}(\mathcal M)=\text{Inn}(\mathcal M)$ and $\mathcal F_s(\mathcal M)=\mathbb N$.
For certain II$_1$ factors $\mathcal M$, including the group II$_1$ factor $\text{L}(\mathbb Z\wr(\mathbb F_2\times\mathbb F_2))$, a description of $\text{End}(\mathcal M)$ and a proof 
that $\mathcal F_s(\mathcal M)\cap (0,1)=\emptyset$ were given in \cite[Corollary E]{Io10}.
Examples of II$_1$ factors $\mathcal M$ with $\text{End}(\mathcal M)=\text{Inn}(\mathcal M)$ and $\mathcal F_s(\mathcal M)=\mathbb N$ were found in the unpublished preprint \cite{De13} via a technically involved construction. Simpler examples of such II$_1$ factors 
were obtained in \cite[Theorem D]{PV21}.  While the II$_1$ factors from \cite{De13,PV21} have the stronger property that all $*$-homomorphisms $\theta\colon \mathcal M\rightarrow \mathcal M^t$ are trivial, none of them arise from groups as in Theorem \ref{endo}.

Answering a question from \cite{Jo83}, it was shown in \cite{PP86b} (see also \cite{Po86}) that $\mathcal I_\mathcal M$ is countable, for any II$_1$ factor $\mathcal M$ having property (T) in the sense of \cite{Co82,CJ85}.
The problem of computing $\mathcal I_{\text{L}(G)}$ for ICC property (T) groups $G$ was proposed  
in \cite[Problem 5]{dH95}. In the spirit of Connes' rigidity conjecture, we conjecture that $\mathcal F_{\text{s}}({\rm L}(G))=\mathbb N$ and therefore
$\mathcal I_{{\rm L}(G)}\subseteq\mathbb N\cup\{\infty\}$ for every ICC property (T) group $G$.
Theorem \ref{rigembed} confirms this conjecture for every property (T) group $G\in\mathcal C_0$. It remains however an open problem to compute explicitly $\mathcal I_{{\rm L}(G)}$, for even a single ICC property (T) group $G$.

\paragraph{Acknowledgments.} The first author has been supported by the NSF Grants  FRG-DMS-1854194 and DMS-2154637, the second author has been supported by NSF Grants FRG-DMS-1854074 and DMS-2153805 and a Simons Fellowship, the third author has been supported by the NSF grant DMS-1853989, and the fourth author has been supported by the AMS--Simons Travel grant (Grant agreement No. IP00672308). 

\paragraph{Data availability statement.} This manuscript has no associated data.

\section {Wreath-like products with only inner endomorphisms}
\label{groupwithnoinjhom}
The goal of this section is to construct examples of property (T) wreath-like products with only inner endomorphisms, see Theorem \ref{Thm:Group}.
In the first part of the section, we collect the results on hyperbolic and relatively hyperbolic groups necessary for proving Theorem \ref{Thm:Group}.


We assume the reader to be familiar with the notion of a hyperbolic group. A group $G$ is \emph{hyperbolic relative to a subgroup} $H$ if there exists a finite relative presentation of $G$ with respect to $H$ with linear relative Dehn function. This definition works for groups of any cardinality. For countable groups, there are other equivalent definitions. For details, we refer to \cite{Hru,Osi06} and the references therein. 

Recall that a subgroup of a hyperbolic group is said to be \textit{elementary} if it is virtually cyclic. By \cite[Lemma 1.16]{Ols93}, every infinite order element $g$ of a hyperbolic group $G$ is contained in a unique maximal elementary subgroup of $G$ denoted by $E(g)$.  Schur's theorem (see \cite[Theorem 5.7]{Isa}) easily implies that every torsion-free virtually cyclic group is cyclic; alternatively, we can refer to a theorem of Stallings \cite{Sta} asserting that finitely generated, torsion-free, virtually free groups are free. Therefore, $E(g)$ is cyclic for any non-trivial element of a torsion-free hyperbolic group $G$.

The result below can be derived from \cite[Theorem 7.11]{Bow} and is stated explicitly in \cite[Corollary 1.7]{Osi06b}. We attribute it to Bowditch since the arXiv version of \cite{Bow}  appeared long before \cite{Osi06b} was written.

\begin{lem}[Bowditch]\label{Thm:Eg}
Let $G$ be a hyperbolic group. For every infinite order element $g\in G$, the group $G$ is hyperbolic relative to $E(g)$.
\end{lem}

We will also need the following result of Olshanskii from \cite{Ols93}.

\begin{lem}[{\cite[Lemma 3.4]{Ols93}}]\label{Lem:Ols}
For every non-trivial, ICC, hyperbolic group $G$, there exists an infinite order element $g\in G$ such that $E(g)=\langle g\rangle$.
\end{lem}

The theorem below summarizes several useful properties of Dehn filling in relatively hyperbolic groups which were established in  \cite{Osi07, DGO,Sun, CIOS1,CIOS3}. 

\begin{thm}[{\cite{Osi07,Sun,DGO,CIOS1,CIOS3}}]\label{Thm:DF}
Let $G$ be a group hyperbolic relative to a subgroup $H$. There exists a finite subset $\mathcal F\subseteq H\setminus\{ 1\}$ such that, for any $N\lhd H$ satisfying $N\cap \mathcal F=\emptyset$, the following hold.

\begin{enumerate}
\item[(a)] We have $H\cap \ll N\rr=N$; equivalently, the natural map $H/N\to G/\ll N\rr$ is injective. Henceforth, we think of $H/N$ as a subgroup of $G/\ll N\rr$.
\item[(b)] The group $G/\ll N\rr$ is hyperbolic relative to $H/N$. In particular, if $H/N$ is hyperbolic, then so is $G/\ll N\rr$.
\item[(c)] Every finite order element of $G/\ll N\rr$ is conjugate to an element of the subgroup $H/N$ in $G/\ll N\rr$ or is an image of a finite order element of $G$ under the natural homomorphism $G\to G/\ll N\rr$.
\item[(d)] We have $G/[\ll N\rr, \ll N\rr]\in \WR (N/[N,N], G/\ll N\rr\curvearrowright I)$, where $I$ is the set of left cosets $G/H\ll N\rr$ and the action of $G/\ll N\rr$ on $I$ is by left multiplication; in particular, the stabilizers of elements of $I$ in $G/\ll N\rr$ are conjugates of $H/N$.
\item[(e)] If $H\ne G$ and $G$ is ICC, then $G/\ll N\rr$ is ICC and non-trivial.
\end{enumerate}
\end{thm}

Parts (a) and (b) of Theorem \ref{Thm:DF} were proved in \cite{Osi07} (see also \cite{GM}). Part (c) is a simplification of \cite[Theorem 7.19 (f)]{DGO}. Part (d) was stated and proved in \cite{Sun} using the language of induced modules; it was reformulated in terms of wreath-like products in \cite[Theorem 2.7]{CIOS1}. Finally, (e) is a particular case of \cite[Proposition 3.29]{CIOS3}.

The next result is well-known (see, for example, \cite[Theorem 34, Chapter 8]{GH}).
\begin{lem}\label{Lem:VCC}
    In every hyperbolic group, the centralizer of any infinite order element is elementary.
\end{lem} 

In general, centralizers of (non-trivial) finite order elements in hyperbolic groups need not be elementary. We introduce the following.

\begin{defn}
    We say that a group $G$ has \textit{virtually cyclic centralizers} (or is a \textit{VCC group}) if $C_G(g)$ is virtually cyclic for all $g\in G\setminus\{ 1\}$.
\end{defn}

Clearly, if a group $G$ is VCC and not virtually cyclic, then $G$ is ICC. Further, from Lemma \ref{Lem:VCC}, we immediately obtain the following.

\begin{cor}\label{Cor:VCC}
    A hyperbolic group $G$ is VCC if and only if the centralizer of any non-trivial, finite order element of $G$ is virtually cyclic. In particular, every torsion-free hyperbolic group is VCC.
\end{cor}

The next lemma shows that the VCC property is inherited by certain quotients of VCC hyperbolic groups.

\begin{lem}\label{Lem:cent}
Let $G$ be a non-elementary, hyperbolic, VCC group. For any infinite order element $g\in G$ such that $\langle g\rangle \lhd E(g)$ and any sufficiently large $n\in \mathbb N$, the group $G/\ll g^n\rr$ is also non-elementary, hyperbolic, and VCC. 
\end{lem}

\begin{proof}
Let $g\in G$ be an infinite order element. In \cite{Ols93}, Olshanskii proved that, for any sufficiently large natural number $n$, the quotient group $\overline G=G/\ll g^n\rr$ satisfies the following.

\begin{enumerate}
\item[(a)] $\overline G$ is hyperbolic and not virtually cyclic;

\item[(b)]  the centralizer of every finite order element of $\overline G$ is virtually cyclic or equals $\overline G$.
\end{enumerate}

\noindent Indeed, (a) is guaranteed by part 7) of \cite[Theorem 3]{Ols93} and (b) follows immediately from Theorem 4 and Proposition 1 in \cite{Ols93} (see the remark after Theorem 4 there; in particular, this remark and the assumption that $G$ satisfies the VCC condition allow us to apply \cite[Theorem 4]{Ols93} in this case.)

Furthermore, by Lemma \ref{Thm:Eg}, $G$ is hyperbolic relative to $E(g)$. Applying Theorem \ref{Thm:DF} to the ICC group $G$, the subgroup $H=E(g)$, and the normal subgroup $N=\langle g^n\rangle$ of $H$, we obtain that

\begin{enumerate}
    \item[(c)] $\overline G$ is ICC 
\end{enumerate} 
for all sufficiently large $n\in \mathbb N$.  

Fix any $n$ such that conditions (a)--(c) hold. By Corollary \ref{Cor:VCC} we only need to consider centralizers of non-trivial finite order elements in $\overline G$. Suppose that $C_{\overline G}(x)$ is not virtually cyclic for some finite order element $x\in \overline{G}\setminus \{ 1\}$. Then $C_{\overline G} (x)=\overline G$ by (b), which contradicts (c). 
\end{proof}

\begin{cor}\label{Cor:VCCquot}
Let $G$ be a non-cyclic, torsion-free, hyperbolic group. For any 
$g\in G\setminus\{1\}$ and any sufficiently large $n\in \mathbb N$, the group $G/\ll g^n\rr$ is non-elementary, hyperbolic, and VCC.
\end{cor}

\begin{proof}
Since $G$ is torsion-free, it is a VCC group. Further, for any non-trivial $g\in G$, $E(g)$ is cyclic as we explained above. In particular, $\langle g\rangle \lhd E(g)$ and Lemma \ref{Lem:cent} applies.
\end{proof}

A sequence of subgroups $\{N_i\}_{i\in \NN}$, of a group $H$ is said to be \emph{cofinal} if, for every finite set $\mathcal F\subseteq H\setminus \{ 1\}$, we have $N_i\cap \mathcal F=\emptyset$ for all but finitely many $i\in \NN$. In particular, if $G$ is hyperbolic relative to $H$ and $\{N_i\}_{i\in \NN}$ is a cofinal sequence of normal subgroups of $H$, then properties (a)--(e) from Theorem \ref{Thm:DF} hold for all but finitely many subgroups $N=N_i$.

The next result is a simplification of \cite[Theorem 3]{DG}. Note that this theorem is proved in \cite{DG} under the assumption that $G$ is \emph{rigid}, i.e., it  does admit non-trivial splittings as an amalgamated free product or an HNN-extension that satisfy certain additional conditions. This assumption is automatically satisfied whenever $G$ has property (T) group since property (T) groups do not admit any non-trivial splittings at all.

\begin{thm}[{\cite{DG}}]\label{Thm:DG}
Let $G$ be a property (T) group hyperbolic relative to a proper subgroup $H$ and let $\{ N_i\} _{i\in \NN}$ be a cofinal set of normal subgroups of $H$. Suppose that, for each $i\in \NN$, there exists an automorphism $\alpha_i\in {\rm Aut}(G/\ll N_i\rr)$ stabilizing (setwise) the set of conjugates of the subgroup $H\ll N_i\rr/\ll N_i\rr$ in $G/\ll N_i\rr$. Then there exists $\alpha\in {\rm Aut} (G)$, an infinite subset $I\subseteq \NN$, and inner automorphisms $\iota _i\in {\rm Inn} (G/\ll N_i\rr)$ such that the diagram
\begin{equation}\label{Eq:cdalpha}
\begin{tikzcd}
G\ar[r, "\alpha"]\ar[d] & G\ar[d]\\
G/\ll N_i\rr \ar[r, "\iota_i\circ \alpha_i"] & G/\ll N_i\rr\\
\end{tikzcd}
\end{equation}
is commutative for all $i\in I$; the vertical arrows are natural homomorphisms.
\end{thm}

\noindent With these preliminaries at hand we are now ready to prove our main group theoretic result.

\begin{thm}\label{Thm:Group}
Let $G$ be a non-trivial, VCC, hyperbolic group with property (T) such that ${\rm Out}(G)=\{ 1\}$. There exists an infinite order element $g\in G$ such that for all sufficiently large $n\in \NN$, the following hold.
\begin{enumerate}
\item[(a)] $G/[\ll g^{n}\rr, \ll g^{n}\rr]\in \WR(\ZZ, B\curvearrowright I)$, where $B=G/\ll g^{n}\rr$, the action of $B$ on $I$ is transitive, and stabilizers of elements of $I$ are isomorphic to $\mathbb Z/n\mathbb Z$.
\item[(b)] $B$ is non-elementary, hyperbolic, and VCC (in particular, ICC). 
\item[(c)] Every injective homomorphism $B\to B$ is an inner automorphism.
\end{enumerate}
\end{thm}

\begin{proof}
By Lemma \ref{Lem:Ols}, there exists an infinite order element $g\in G$ such that $E(g)=\langle g\rangle$. By Theorem \ref{Thm:Eg}, $G$ is hyperbolic relative to $E(g)$. For every $n\in \NN$, let $N_n=\langle g^n\rangle \leqslant E(g)$. We first show that $Out(G/\ll N_n\rr)=\{ 1\}$ for all but finitely many $n\in \NN$.

We argue by contradiction. Assume that there exists an infinite subset $I_1\subseteq \NN$ such that, for every $i\in I_1$, there is a non-inner automorphism $\alpha_i\in {\rm Aut}(G/\ll N_i\rr)$. Clearly, $\{ N_i\}_{i\in \NN}$ is a cofinal sequence of normal subgroups of $E(g)$. Therefore, there exists a cofinite subset $I_2\subseteq \NN$ such that the conclusion of Theorem \ref{Thm:DF} holds for $H=E(g)$ and $N=N_i$ whenever $i\in I_2$. In particular, we have $$E(g)\ll N_i\rr/\ll N_i\rr\cong E(g)/N_i=\langle g\rangle /\langle g^i\rangle \cong \mathbb Z/i\mathbb Z
$$ for all $i\in I_2$.

Recall that finite orders of elements of every hyperbolic group are uniformly bounded \cite{BG}. Thus, there exists  $N\in \NN$ such that for every finite order element $f\in G$, we have $|f|\le N$. Let
$$
J=\{ j\in I_1\cap I_2\mid j> N\}.
$$
Clearly, $J$ is infinite in $\NN$ and, in particular, infinite. Henceforth, the index $j$ will be assumed to range in $J$ by default.

Let $\e_j\colon G\to G/\ll N_j\rr$ denote the natural homomorphism. By parts (a) and (c) of Theorem~\ref{Thm:DF}, the only cyclic subgroups of order $j$ in $G/\ll N_j\rr$ are conjugates of $\langle \e_j(g)\rangle $. Hence, each $\alpha_j$ maps the subgroup $\langle \e_j(g)\rangle $ to its conjugate in $G/\ll N_j\rr$. By Theorem \ref{Thm:DG}, there exists $\alpha\in {\rm Aut} (G)$ and inner automorphisms $\iota _j\in {\rm Inn} (G/\ll N_j\rr)$ such that the diagram (\ref{Eq:cdalpha}) is commutative for infinitely many $j\in J$. Since ${\rm Out}(G)=\{ 1\}$, $\alpha $ is an inner automorphism of $G$. It follows that $\alpha _j\in {\rm Inn} (G/\ll N_j\rr)$ for infinitely many $j\in J$, which contradicts the choice of $\alpha_j$. This contradiction shows that ${\rm Out}(G/\ll N_n\rr)=\{ 1\}$ for infinitely many $n\in \NN$.

Note that $N_n/[N_n, N_n]\cong\ZZ$ and $\ll N_n\rr =\ll g^n\rr$ for all $n\in \NN$. Using Theorem \ref{Thm:DF}, we obtain (a) for all sufficiently large $n$. By Lemma \ref{Lem:cent}, we can also ensure (b) by taking $n$ sufficiently large. Thus, we have conditions (a), (b) and ${\rm Out}(G/\ll N_n\rr)=\{ 1\}$ for all sufficiently large $n\in \NN$. Generalizing a result of Sela \cite{Sel}, Moioli \cite{Moi} proved that every $1$-ended hyperbolic group is co-Hopfian, i.e., every injective homomorphism is an automorphism (for groups with property (T), this also follows from \cite[Theorem 4.4]{CG}). Since ${\rm Out}(G/\ll N_n\rr)=\{ 1\}$, we obtain (c).
\end{proof}



\begin{lem}\label{trivialout}
There exists a non-elementary, VCC,  hyperbolic group $B$ with property (T), trivial abelianization, and ${\rm Out}(B)=\{ 1\}$.
\end{lem}

\begin{proof}
Let $L$ be a uniform lattice in $Sp(n,1)$. By Selberg's lemma, there exists a finite index torsion-free subgroup $H\le L$. Being a finite index subgroup of $L$, the group $H$ is also a uniform lattice in $Sp(n,1)$ and, therefore, is hyperbolic and has property (T). 

As shown in \cite[Corollary~3.24]{CIOS3}, $H$ has a non-elementary, torsion-free, hyperbolic quotient group $\overline H$ such that $\overline H=[\overline H,\overline H]$. Further, by \cite[Theorem 4.5]{CIOS2}, ${\overline H}$ has a non-elementary hyperbolic quotient group $B$ such that ${\rm Out}(B)=\{ 1\}$. 
Obviously, $B$ has property (T) and trivial abelianization. It remains to prove that $B$ is a VCC group. The latter property is not stated explicitly in \cite[Theorem 4.5]{CIOS2} but can be easily extracted from its proof as explained below. In the next paragraph, ``the proof" refers to the proof of Theorem 4.5 in \cite{CIOS2}.

First note that, in the notation of \cite{CIOS2}, we apply Theorem 4.5 in the case $Q=\{ 1\}$ here and so $\overline G=B$. The group $\overline G=B$ constructed in the proof contains elements $s$ and $t$ of order $2$ and $3$, respectively. As explained in the last two sentences of the fourth paragraph of the proof, every finite order element of $B$ is conjugate to one of the elements $1$, $s$, $t$, $t^{-1}$, and the centralizers of elements $s$, $t$, and $t^{-1}$ are virtually cyclic (see equations (32) and (38) in \cite{CIOS2}).  Therefore, 
the centralizer of every non-trivial, finite order element of $B$ is virtually cyclic. Combining this with Corollary \ref{Cor:VCC}, we obtain that $B$ is a VCC group.  
\end{proof}

\section{Preliminaries on von Neumann algebras}\label{prelimvN}
In this section, we recall some terminology concerning von Neumann algebras along with several results, mainly from \cite{CIOS1}, which will be needed in the proof of Theorem \ref{main}. 

\subsection{Tracial von Neumann algebras}\label{tracialvN}

A {\it tracial von Neumann algebra} is a pair $(\M,\tau)$ consisting of a von Neumann algebra $\M$  and a {\it trace} $\tau$, i.e., a normal  faithful tracial state $\tau:\M\rightarrow\mathbb C$.  We denote by L$^2(\M)$ the Hilbert space obtained as the closure of $\M$ with respect to the $2$-norm given by $\|x\|_2=\sqrt{\tau(x^*x)}$. We denote by $\sU(\M)$ the group of {\it unitaries}  of $\M$ and by $(\M)_1=\{x\in \M\mid \|x\|\leq 1\}$ the {\it unit ball} of $\M$. 
We will always assume that $\M$ is {\it separable} or, equivalently, that L$^2(\M)$ is a separable Hilbert space. We denote by $\text{Aut}(\M)$ the group of $\tau$-preserving automorphisms of $\M$.  For $u\in\sU(\M)$,  the {\it inner} automorphism $\text{Ad}(u)\in\text{Aut}(\M)$ is given by $\text{Ad}(u)(x)=uxu^*$.
By von Neumann's bicommutant theorem, if $X\subset \M$ is a set closed under adjoint which contains the identity, then $X''$ is the von Neumann algebra generated by $X$.

The tracial von Neumann algebra $(\M,\tau)$  is called {\it amenable} if there exists a sequence $\xi_n\in \text{L}^2(\M)\otimes \text{L}^2(\M)$ such that $\langle x\xi_n,\xi_n\rangle\rightarrow\tau(x)$ and $\|x\xi_n-\xi_nx\|_2\rightarrow 0$, for every $x\in \M$.


Let $\Q\subset \M$ be a von Neumann subalgebra, which we will always assume to be unital. 
We denote by $\Q'\cap \M=\{\text{$x\in \M\mid xy=yx$, for all $y\in \Q$}\}$ the {\it relative commutant} of $\Q$ in $\M$ and by $\sN_\M(\Q)=\{u\in\sU(\M)\mid u\Q u^*=\Q\}$ the {\it normalizer} of $\Q$ in $\M$. The {\it center} of $\M$ is given by $\sZ(\M)=\M'\cap \M$. 
A {\it Cartan subalgebra} $\Q\subset\M$ is a maximal abelian von Neumann algebra such that $\sN_\M(\Q)''=\M$. 

 Following \cite[Proposition 4.1]{Po01}, we say that $\Q\subset \M$ has the {\it relative property (T)} if for every $\varepsilon>0$, we can find a finite set $F\subset \M$ and $\delta>0$ such that if $\mathcal H$ is an $\M$-bimodule and $\xi\in\mathcal H$ satisfies $\|\langle\cdot\xi,\xi\rangle-\tau(\cdot)\|\leq\delta,\|\langle\xi\cdot,\xi\rangle-\tau(\cdot)\|\leq\delta$ and $\|x\xi-\xi x\|\leq\delta$, for every $x\in F$, then there is $\eta\in\mathcal H$ such that $\|\eta-\xi\|\leq\varepsilon$ and $y\eta=\eta y$, for every $y\in \Q$.
For instance, $\Q\subset\M$ has the relative property (T) whenever $\Q=\text{L}(G)$, for a countable property (T) group $G$.


We next record a remark mentioned in the introduction.

\begin{rem}\label{index}
    Let $\M$ be a II$_1$ factor. Let $t\in \mathcal I_\M\cap (0,\infty)$. Then there exists a subfactor $\Nn\subset \M$ with $[\M:\Nn]=t$. Since $e_{\Nn}\langle\M,e_{\Nn}\rangle e_{\Nn}\cong\Nn$ and $\langle\M,e_{\Nn}\rangle$ is a II$_1$ factor, we have that $\langle\M,e_{\Nn}\rangle\cong \Nn^{t}$.  Since we have a unital $*$-homomorphism $\M\rightarrow\langle\M,e_{\Nn}\rangle$,
 we obtain a unital $*$-homomorphism $\M\rightarrow\Nn^t\subset\M^t$. Thus, $t\in\mathcal F_{\text{s}}(M)$, which proves that $\mathcal I_\M\subset\mathcal F_{\text{s}}(M)\cup\{\infty\}$.
\end{rem}

\subsection {Intertwining-by-bimodules} We recall from  \cite [Theorem 2.1, Corollary 2.3]{Po03} Popa's {\it intertwining-by-bimodules} theory.
\begin{thm}[\cite{Po03}]\label{corner} Let $(\M,\tau)$ be a tracial von Neumann algebra and $\P\subset p\M p, \Q\subset q\M$q be von Neumann subalgebras.
Then the following conditions are equivalent:

\begin{enumerate}

\item There exist projections $p_0\in \P, q_0\in \Q$, a $*$-homomorphism $\theta:p_0\P p_0\rightarrow q_0\Q q_0$  and a nonzero partial isometry $v\in q_0\M p_0$ such that $\theta(x)v=vx$, for all $x\in p_0\P p_0$.

\item There is no sequence $u_n\in\sU(\P)$ satisfying $\|{\rm  E}_\Q(x^*u_ny)\|_2\rightarrow 0$, for all $x,y\in p\M$.

\end{enumerate}

If these conditions hold true,  we write $\P\prec_{\M}\Q$.
Moreover, if $\P p'\prec_{\M}\Q$ for any nonzero projection $p'\in \P'\cap p\M p$, we write $\P\prec^{\rm s}_{\M}\Q$.
\end{thm}

The following is a corollary of the structure of normalizers in crossed products arising from actions of hyperbolic  groups \cite[Theorem 1.4]{PV12} observed in \cite[Theorem 3.10]{CIOS1}.

\begin{thm}\label{relativeT}

Let $G,H$ be countable groups, with $H$ a non-elementary subgroup of a hyperbolic group. Let $\delta:G\rightarrow H$ be a homomorphism and $G\curvearrowright (\Q,\tau)$ a trace preserving action on a tracial von Neumann algebra  $(\Q,\tau)$. Let $\M=\Q\rtimes G$ and $\P\subset p\M p$ be an amenable von Neumann subalgebra.  
Assume that there is a von Neumann subalgebra $\R\subset \sN_{p\M p}(\P)''$ with the relative property (T) such that $\R\nprec_{\M}\Q\rtimes\ker(\delta)$. Then  $\P\prec_{\M}^{\rm s}\Q\rtimes\ker(\delta)$.

\end{thm}

\begin{rem}
    We also record the following fact which we will use implicitly later on. In the context of Theorem \ref{corner}, if $\P$ is diffuse and has property (T) (i.e.,  $\P\subset \P$ has the relative property (T)) and $\Q$ is amenable, then $\P\nprec_\M \Q$. This in particular implies that if $G,H$ are subgroups of a countable group $K$ such that $G$ is infinite and has property (T) and $H$ is amenable, then there does not exist a finite set $F\subset K$ with $G\subset FHF$. Indeed, otherwise we would have that $\text{L}(G)\prec_{\text{L}(K)}\text{L}(H)$, which contradicts the above fact.
\end{rem}

\subsection{Cocycle superrigidity results}
Let $G$ be a countable group and $G\curvearrowright^\sigma (X,\mu)$ be a probability measure preserving ({\it p.m.p.}) action on a standard probability space $(X,\mu)$. 
We equip $\text{L}^{\infty}(X)$ with the trace given by integration against $\mu$ and denote also by $\sigma$ the trace preserving action $G \curvearrowright^{\sigma}\text{L}^{\infty}(X)$ given by $\sigma_g(f)(x)=f (g^{-1} x)$ for every $g\in G$, $x\in X$ and $f\in {\rm L}^\infty(X)$.
From now, we will not distinguish between $\sigma$ and the associated trace preserving action.

Let $H$ be a Polish group. A measurable map $c:G\times X\rightarrow H$ is called a {\it 1-cocycle for $\sigma$} if it satisfies $c(g_1g_2,x)=c(g_1,g_2x)c(g_2,x)$, for all $g_1,g_2\in G$ and almost every $x\in X$.
Two $1$-cocycles $c,c':G\times X\rightarrow H$ are {\it cohomologous} if there is a measurable map $\varphi:X\rightarrow H$ such that $c'(g,x)=\varphi(gx)c(g,x)\varphi(x)^{-1}$, for all $g\in G$ and almost every $x\in X$. Let $\mathcal V$ be a family of Polish groups. Following \cite[5.6.0]{Po05},  we say that $\sigma$ is $\mathcal V$-{\it cocycle superrgid} if for any $H\in\mathcal V$, any $1$-cocycle $c:G\times X\rightarrow H$ is cohomologous of a homomorphism $\delta:G\rightarrow H$.
By \cite[Definition 2.5]{Po05}, we denote by $\sU_{\text{fin}}$ the family of Polish groups which are isomorphic to a closed subgroup of $\sU(\M)$, for some separable tracial von Neumann algebra $(\M,\tau)$.

Let $G\curvearrowright^\sigma (Y,\nu)^I$ be a p.m.p. action, where $(Y,\nu)$ is a standard probability space and $I$ is a countable set endowed with a $G$-action. 
Following \cite[Definition 2.5]{KV15}, we say that $\sigma$ is {\it built over} $G\curvearrowright I$ if  $\sigma_g(\text{L}^\infty(Y)^{i})=\text{L}^\infty(Y)^{g i}$, for all $g\in G$ and $i\in I$. Equivalently, there is $\alpha:G\times I\rightarrow\text{Aut}(Y,\nu)$ such that $(\sigma_g(y))_i=\alpha(g,i)(y_{g^{-1} i})$, for all $g\in G$ and $y=(y_i)_{i\in I}\in Y^I$.
For example, the generalized Bernoulli action $G\curvearrowright (Y,\nu)^I$ is clearly built over $G\curvearrowright I$.

For us it will be important that,   as observed in \cite[Lemma 3.4 and Remark 3.5]{CIOS1},  such actions  occur naturally  in the setting of wreath-like product groups:

\begin{lem}[\cite{CIOS1}]\label{bover} Let $A,B$ be countable groups and $G\in\W\R(A,B\curvearrowright I)$, where $A$ is abelian and $B\curvearrowright I$ is an action on a countable set $I$. Let $\varepsilon:G\rightarrow B$ be the quotient homomorphism and $(u_g)_{g\in G}$ the canonical generating unitaries of $\emph{L}(G)$.
Let $G\curvearrowright I$ and $G\curvearrowright^\sigma \emph{L}(A^{(I)})$ be the action and the trace preserving action given by $g\cdot i=\varepsilon(g)i$ and $\sigma_g=\emph{Ad}(u_g)$, for every $g\in G$ and $i\in I$. Let $B\curvearrowright^\alpha \emph{L}(A^{(I)})$ be the trace preserving action given by $\alpha_g=\sigma_{\widehat{g}}$, where $\widehat{g}\in G$ is any element such that $\varepsilon(\widehat{g})=g$, for every $g\in B$.

Then $\sigma$ is built over $G\curvearrowright I$ and $\alpha$ is built over $B\curvearrowright I$.
\end{lem}

As noted in \cite[Theorem 3.6]{CIOS1}, actions built over actions satisfy Popa's cocycle superrigidity theorem \cite{Po05}. In particular, we have the following.

\begin{thm}[\cite{CIOS1}]\label{CS}
Let $G$ be a countable group with property (T) and  $G\curvearrowright I$ be an action on a countable set $I$ with infinite orbits.

Then any p.m.p. action $G\curvearrowright^{\sigma}(Y,\nu)^I$ built over $G\curvearrowright I$ is $\sU_{\emph{fin}}$-cocycle superrigid.
\end{thm}

\begin{proof}
Let $H\in\sU_{\text{fin}}$ and $c:G\times Y^I\rightarrow H$ be a $1$-cocycle. Then $H$ is isomorphic to a closed subgroup of $\sU(\M)$, for some separable tracial von Neumann algebra $(\M,\tau)$.
By applying \cite[Theorem 3.6]{CIOS1}, when viewed as a $\sU(\M)$-valued cocycle, $c$ is cohomologous to a homomorphism $G\rightarrow\sU(\M)$. Since $\sigma$ is weakly mixing, \cite[Proposition 3.5]{Po05} implies that $c$ is cohomologous to a homomorphism $G\rightarrow H$, as desired.
\end{proof}

The following consequence of Theorem \ref{CS} was proved in \cite[Proposition 4.27]{CIOS3}.

 \begin{cor}[\cite{CIOS3}]\label{injectiveOut}
Let $A,B$ be countable groups and $B\curvearrowright I$ an action with infinite orbits. Assume that $G\in\mathcal W\mathcal R(A,B\curvearrowright I)$ has property (T).  Let $\delta:G\rightarrow G$ be a homomorphism such that $\delta(A^{(I)})\subset A^{(I)}$.
Identify $B=G/A^{(I)}$ and let $\rho:B\rightarrow B$ be the homomorphism given by $\rho(gA^{(I)})=\delta(g)A^{(I)}$.   If $\rho\in \emph{Inn}(B)$, then $\delta\in \emph{Inn}(G)$.
\end{cor}

\subsection{Cartan subalgebras}

For further use, we also recall  two conjugacy results for Cartan subalgebras from \cite{Io10}. These results are stated and proved as such in \cite[Lemmas 3.7 and 3.8]{CIOS1}.

\begin{lem}[\cite{Io10}]\label{conj1}
Let $\M$ be a II$_1$ factor, $\A\subset \M$ be a Cartan subalgebra and $\D\subset \M$ be an abelian von Neumann subalgebra. Let $\C=\D'\cap \M$ and assume that $\C\prec_{\M}^{\rm s}\A$. Then there exists $u\in\sU(\M)$ such that $\D\subset u\A u^*\subset \C$.
\end{lem}

\begin{lem}[\cite{Io10}]\label{conj2}
Let $\M$ be a II$_1$ factor, $\A\subset \M$ a Cartan subalgebra, $\D\subset \M$ an abelian von Neumann subalgebra and let $\C=\D'\cap \M$. Assume that $\C\prec_{\M}^{\text{s}}\A$ and $\D\subset \A\subset \C$.
Let $(\alpha_g)_{g\in G}$ be an action of a group $G$ on $\C$ such that $\alpha_g=\emph{Ad}(u_g)$, for some $u_g\in\sN_\M(\D)$, for every $g\in G$.
 Assume that the restriction of the action $(\alpha_g)_{g\in G}$  to $\D$ is free.
Then there is an action $(\beta_g)_{g\in G}$ of $G$ on $\C$ such that
\begin{enumerate}
\item\label{unu} for every $g\in G$ we have that $\beta_g=\alpha_g\circ\emph{Ad}(\omega_g)=\emph{Ad}(u_g\omega_g)$, for some $\omega_g\in\sU(\C)$, and
\item\label{doi} $\A$ is $(\beta_g)_{g\in G}$-invariant and the restriction of
 $(\beta_g)_{g\in G}$ to $\A$ is free.
\end{enumerate}
\end{lem}


\subsection{Strong rigidity for orbit equivalence embeddings}\label{STRONG}

The following generalization of \cite[Theorem 0.5]{Po04} to a large class of actions built over was obtained in \cite[Theorem 4.1]{CIOS1}.
If $\sR$ is a countable p.m.p. equivalence relation on a standard probability space $(X,\mu)$, then the {\it full group} $[\sR]$ consists of all automorphisms $\theta$ on $(X,\mu)$ such that $(\theta(x),x)\in\sR$, for almost every $x\in X$. If $G\curvearrowright (X,\mu)$ is a p.m.p. action, then  $\sR(G\curvearrowright X)=\{(x_1,x_2)\in X\times X\mid G\cdot x_1=G\cdot x_2\}$ is its {\it orbit equivalence relation}.
\begin{thm}[\cite{CIOS1}]\label{SOE}
Let $ D$ be an ICC group and $ D\curvearrowright^{\alpha}(X,\mu)=(Y^I,\nu^I)$ be a measure preserving action built over an action $D\curvearrowright I$, where $(Y,\nu)$ is a probability space.  Let $\widetilde D= D\times\mathbb Z/n\mathbb Z$ and $(\widetilde X,\widetilde\mu)=(X\times\mathbb Z/n\mathbb Z,\mu\times c)$, where $n\in\mathbb N$ and $c$ is the counting measure of $\mathbb Z/n\mathbb Z$.
Consider the action $\widetilde D\curvearrowright^{\widetilde\alpha} (\widetilde X,\widetilde\mu)$ given by $(g,a)\cdot (x,b)=(g\cdot x,a+b)$.

Let $B$ be a countable group with a normal subgroup $B_0$ such that the pair $(B,B_0)$ has the relative property (T).
Let $X_0\subset \widetilde X$ be a measurable set  and $B\curvearrowright^{\beta} (X_0,\widetilde\mu_{|X_0})$ be a weakly mixing, free, measure preserving action such that  $B\cdot x\subset\widetilde D\cdot x$, for almost every $x\in X_0$.
 Assume that for every $i\in I$ and  $g\in D\setminus\{1\}$, there is a sequence $(h_m)\subset B_0$ such that  for every  $s,t\in\widetilde D$ we have \begin{equation}\begin{split}
 &\lim_{m\ra \infty}\widetilde\mu(\{x\in X_0\mid h_m\cdot x\in s(\emph{Stab}_D(i)\times\mathbb Z/n\mathbb Z)t\cdot x\})= 0,\text{ and }\\
&\lim_{m\ra \infty}\widetilde\mu(\{x\in X_0\mid h_m\cdot x\in s(\emph{C}_D(g)\times\mathbb Z/n\mathbb Z)t\cdot x\})= 0.\end{split}
\end{equation}

Then one can find a subgroup $D_1< D$,  a group isomorphism $\delta:B\rightarrow D_1$ and $\theta\in [\sR(\widetilde D\curvearrowright\widetilde X)]$ such that
 $\theta(X_0)=X\equiv X\times\{0\}$ and
$\theta\circ\beta(h)=\alpha(\delta(h))\circ\theta$, for every $h\in B$.
\end{thm}

Here,  $\text{Stab}_D(i)=\{h\in D\mid hi=i\}$ and $\text{C}_D(g)=\{h\in D\mid gh=hg\}$, for $i\in I$ and $g\in D$.

The following result due to Popa \cite{Po04} (see \cite[Lemma 4.4]{CIOS1}) shows  that the weakly mixing condition from Theorem \ref{SOE} is automatically satisfied after passing to an ergodic component of a finite index subgroup.

\begin{lem}[\cite{Po04}]\label{wmix}
Assume the setting of Theorem \ref{SOE}.  Then one can find a finite index subgroup $ S< B$ and a $\beta(S)$-invariant non-null measurable set $Y\subset X_0$ such that $\widetilde\mu(\beta(h)(Y)\cap Y)=0$, for every $h\in B\setminus S$, and the restriction of $\beta_{| S}$ to $Y$ is weakly mixing.

Moreover, there exist a partition $X_0=\sqcup_{i=1}^lX_i$ into non-null measurable sets,  for some  $l\in\mathbb N\cup\{\infty\}$,  and a finite index subgroup $ S_i< B$ such that $X_i$ is $\beta( S_i)$-invariant and
 the restriction of $\beta_{| S_i}$ to $X_i$ is weakly mixing, for all $i$.
\end{lem}

\begin{proof}
The main assertion is \cite[Lemma 4.4]{CIOS1}. 

For the moreover assertion, note that the proof of \cite[Lemma 4.4]{CIOS1} provides a completely atomic, $\beta(B)$-invariant  von Neumann subalgebra $\mathcal P\subset\text{L}^\infty(X_0)$ such that for every non-zero projection $p\in\mathcal P$, the subgroup $S=\{h\in B\mid\beta(h)(p)=p\}$ of $B$ has finite index in $B$ and the restriction of $\beta_{|S}$ to $\text{L}^\infty(X_0)p$ is weakly mixing.
Write $\mathcal P=\bigoplus_{i=1}^l\mathbb Cp_i$, where $(p_i)_{i=1}^l$ are non-zero projections. For every $i$, let $X_i\subset X_0$ be a measurable subset such that 
$p_i=\textbf{1}_{X_i}$  and put $S_i=\{h\in B\mid\beta(h)(p_i)=p_i\}$. 
Then the conclusion follows. 
\end{proof}

\subsection{A solidity result}\label{SOL}
Finally, we also recall the following solidity result established in \cite[Corollary 4.7]{CIOS1}.
\begin{cor}[\cite{CIOS1}]\label{solidity}
Let $G\in\mathcal W\mathcal R(A,B\curvearrowright I)$, where $A$ is an abelian group and
$B\curvearrowright I$ an action such that $\emph{Stab}_{B}(i)$ is amenable for every $i\in I$ and $\{i\in I\mid g\cdot i\not=i\}$ is infinite for every $g\in B\setminus\{1\}$.
Let $\Q\subset  p(\emph{L}(G)\overline{\otimes}\mathbb M_n(\mathbb C)) p$ be a diffuse von Neumann subalgebra, for some $n\in\mathbb N$, such that $\Q\prec^{\emph{s}}_{\emph{L}(G)\overline{\otimes}\mathbb M_n(\mathbb C)}\emph{L}(A)\overline{\otimes}\mathbb M_n(\mathbb C)$.

Then $\Q'\cap p(\emph{L}(G)\overline{\otimes}\mathbb M_n(\mathbb C))p$ is amenable.
\end{cor}

 \section{Rigidity for second cohomology groups}
The goal of this section is to establish the following theorem, which strengthens a result of Jiang \cite{Ji15}, and will be needed in the proof of Theorem \ref{rigembed}.

 We start by recalling some terminology.
Let  $G\curvearrowright^{\sigma} (X,\mu)$ be a p.m.p. action.
 A map $c: G\times G\rightarrow\sU(\text{L}^{\infty}(X))$ is a {\it $2$-cocycle for $\sigma$} if it satisfies  $c_{g,h}c_{gh,k}=\sigma_g(c_{h,k})c_{g,hk}$, for every $g,h,k\in G$, and a {\it $2$-coboundary for $\sigma$} if there is a map $b: G\rightarrow\sU(\text{L}^{\infty}(X))$ such that $c_{g,h}=b_{gh}^*b_g\sigma_g(b_h)$, for every $g,h\in G$. If $n\in\mathbb N$, then $G\curvearrowright^{\sigma^{\otimes_n}}(X,\mu)^n$ denotes the direct product of $n$ copies of $\sigma$.
 

For a compact Polish abelian group $X$ endowed with its Haar measure $\mu$, we view its dual $\widehat{X}$ as a subgroup of  $\sU(\text{L}^{\infty}(X))$ and  identify  L$^{\infty}(X)=$ L$(\widehat{X})$, via the Fourier transform. 

\begin{thm}[\cite{Ji15}]\label{jiang} 
Let $ G$ be a countable group and $ G\curvearrowright^{\sigma} (X,\mu)$ be a p.m.p. action by continuous automorphisms on a compact Polish abelian group $X$ endowed with its Haar measure $\mu$. Assume that $\sigma$ is weakly mixing and the actions $\sigma^{\otimes_2}$ and $\sigma^{\otimes_4}$ are $\{\mathbb T\}$-cocycle superrigid. Let $c: G\times G\rightarrow\sU(\emph{L}^{\infty}(X))$ be a $2$-cocycle for $\sigma$ such that $c_{g,h}\in\widehat{X}$, for every $g,h\in G$. Assume there exist a  map $w: G\rightarrow\sU(\emph{L}^{\infty}(X))$ and a scalar $2$-cocycle $d: G\times G\rightarrow\mathbb T$ such that $c_{g,h}=d_{g,h}w_{gh}^*w_g\sigma_g(w_h)$, for every $g,h\in G$.

Then there are maps
 $c: G\rightarrow\widehat{X}$, $e: G\rightarrow\mathbb T$ and $z\in\sU(\emph{L}^{\infty}(X))$ such that we have
$c_{g,h}=c_{gh}^*c_g\sigma_g(c_h)$, $d_{g,h}={e_{gh}}\overline{e_g}\overline{e_h}$ and $w_g=e_gc_gz^*\sigma_g(z)$, for every $g,h\in G$. \end{thm}

Theorem \ref{jiang} improves \cite[Theorem 1.1]{Ji15}, where the same statement was proved under the additional assumption that the second cohomology group $\text{H}^2( G,\widehat{X})$
 has no  elements of order $2$ (see \cite[Remark 1.2]{Ji15}).
 This improvement is needed in the proof of Theorem \ref{rigembed} where we apply Theorem \ref{jiang} to
 $X=\widehat{A}^I$, where $A$ is a countable abelian group and $\sigma$ is an action built over $G\curvearrowright I$.  Indeed, if $A$ is finite, then $\text{H}^2( G,\widehat{X})=\text{H}^2( G,A^I)$ is a torsion group that may have elements of order $2$ and so \cite[Theorem 1.1]{Ji15} does not apply.

\begin{proof}
Let $\Delta:\text{L}(\widehat{X})\rightarrow \text{L}(\widehat{X})\overline{\otimes}\text{L}(\widehat{X})$ be the $*$-homomorphism given by $\Delta(u_a)=u_a\otimes u_a$, for $a\in\widehat{X}$ \cite{PV09}. Since $\sigma^{\otimes_2}_g\circ\Delta=\Delta\circ\sigma_g$ and $\Delta(c_{g,h})=c_{g,h}\otimes c_{g,h}$ we derive that $\xi_g:=\Delta(w_g)^*(w_g\otimes w_g)\in\sU(\text{L}(\widehat{X})\overline{\otimes}\text{L}(\widehat{X}))$ satisfies $\xi_{gh}=d_{g,h}\xi_g\sigma_g^{\otimes_2}(\xi_h)$, for every $g,h\in G$. Since $\sigma^{\otimes_4}$ is $\{\mathbb T\}$-cocycle superrigid, \cite[Lemma 2.9]{Ji15} implies that $d$ is a $2$-coboundary, so there is $d: G\rightarrow\mathbb T$ such that $d_{g,h}=\overline{d_{gh}}{d_g}{d_h}$, for every $g,h\in G$. Thus, $(d_g\xi_g)_{g\in G}$ is a $\mathbb T$-valued $1$-cocycle for $\sigma^{\otimes_2}$. Since $\sigma^{\otimes_2}$ is $\{\mathbb T\}$-cocycle superrigid, we get a character $\eta: G\rightarrow\mathbb T$ and $u\in\sU(\text{L}(\widehat{X})\overline{\otimes}\text{L}(\widehat{X}))$ such that $d_g\xi_g={\eta}_gu\sigma_g^{\otimes_2}(u)^*$. So, if $e_g:=\overline{d_g}\eta_g\in\mathbb T$, then
\begin{equation}\label{comult}
\text{$\Delta(w_g)=\overline{e_g}(w_g\otimes w_g)u^*\sigma_g^{\otimes_2}(u)$, for every $g\in G$.}
\end{equation}
Let $\zeta$ be the flip automorphism of $\text{L}(\widehat{X})\overline{\otimes}\text{L}(\widehat{X})$. Since $\zeta\circ\Delta=\Delta$ and $\zeta\circ\sigma_g^{\otimes_2}=\sigma_g^{\otimes_2}\circ\zeta$, \eqref{comult} implies that $\sigma_g^{\otimes_2}(u^*\zeta(u))=u^*\zeta(u))$, for every $g\in G$. Since $\sigma$ is weakly mixing, we get that $u^*\zeta(u)=\eta$, for some $\eta\in\mathbb T$. Since $(\Delta\otimes\text{Id})\circ\Delta=(\text{Id}\otimes\Delta)\circ\Delta$ and $\sigma^{\otimes_2}_g\circ\Delta=\Delta\circ\sigma_g$,  \eqref{comult} also implies that $U=(\Delta\otimes{\text{Id}})(u)(u\otimes 1)(1\otimes u)^*(\text{Id}\otimes\Delta)(u)^*$ satisfies $\sigma_g^{\otimes_3}(U)=U$, for every $g\in G$. Since $\sigma$ is weakly mixing, we also get that $U=\eta'$, for some $\eta'\in\mathbb T$.

By applying \cite[Theorem 3.3]{IPV10} we get that $\eta=\eta'=1$ and  $u=\Delta(z)(z\otimes z)^*$, for some $z\in \sU(\text{L}(\widehat{X}))$.
If $g\in G$, then \eqref{comult} gives that $\Delta(w_gz\sigma_g(z)^*)=\overline{e_g}(w_gz\sigma_g(z)^*)\otimes (w_gz\sigma_g(z)^*)$, which by \cite[Lemma 7.1]{IPV10} gives that $c_g:=\overline{e_g}w_gz\sigma_g(z)^*\in\widehat{X}$.
As $\eta$ is a character, we get that $e_{gh}\overline{e_g}\overline{e_h}=\overline{d_{gh}}d_gd_h=d_{g,h}$ and $c_{gh}^*c_g\sigma_g(c_h)=e_{gh}\overline{e_g}\overline{e_h}w_{gh}^*w_g\sigma_g(w_h)=c_{g,h}$, for $g,h\in G$.
\end{proof}

\section{Rigidity for embeddings of wreath-like product II$_1$ factors}\label{embeddings}

This section is devoted to proving the main result of the paper, Theorem \ref{main}. In fact, we will establish the following strengthening of Theorem \ref{main}.

\begin{thm}\label{rigembed}
Let $G$ be a countable ICC group with property (T) which has a normal abelian group $A$ such that the set $\{aga^{-1}\mid a\in A\}$ is infinite, for every $g\in G\setminus A$.

Let $H\in\mathcal C_0$. Specifically, assume that $H\in\mathcal W\mathcal R(C,D\curvearrowright I)$, where $C$ is a nontrivial abelian group, $D$ is a nontrivial ICC subgroup of a hyperbolic group and $D\curvearrowright I$ is an  action 
such that $\emph{C}_{ D}(g)$ is virtually cyclic and $\emph{Stab}_D(i)$ is amenable for every $g\in D\setminus\{1\}$ and  $i\in I$.

Let $\theta:\emph{L}(G)\rightarrow \emph{L}(H)^t$ be a unital $*$-homomorphism, for some $t>0$. Denote by $(u_g)_{g\in G}$ and $(v_h)_{h\in H}$ the canonical generating unitaries of $\emph{L}(G)$ and $\emph{L}(H)$, respectively.

Then $t\in\mathbb N$ and there are $t_1,\ldots, t_q\in\mathbb N$ with $t_1+\cdots+t_q=t$, for some $q\in\mathbb N$, a finite index subgroup $K<G$ with $A<K$, an injective homomorphism $\delta_i:K\rightarrow H$
such that $\delta_i(A)\subset C^{(I)}$, and a unitary representation $\rho_i:K\rightarrow\sU_{t_i}(\mathbb C)$, for every $1\leq i\leq q$, and a unitary $w\in \emph{L}(H)^t=\emph{L}(H)\overline{\otimes}\mathbb M_t(\mathbb C)$ such that 
$$\text{$w\theta(u_g)w^*=\emph{diag}(v_{\delta_1(g)}\otimes\rho_1(g),
\ldots, v_{\delta_q(g)}\otimes\rho_q(g))
$, for every $g\in K$.}$$

Moreover, when $t=1$, one can find an injective homomorphism $\delta:G\rightarrow H$, a character $\rho:G\rightarrow\mathbb T$ and $w\in\sU(\emph{L}(H))$ so that $\delta(A)\subset C^{(I)}$ and $w\theta(u_g)w^*=\rho(g)v_{\delta(g)}$, for all $g\in G$.

\end{thm}

\begin{proof} Let $B=G/A$. Let $\pi:G\rightarrow B$ and $\chi:H\rightarrow D$ be the quotient homomorphisms. For $g\in B$ and $h\in D$, fix $\widehat{g}\in G$ and $\widehat{h}\in H$ such that $\pi(\widehat{g})=g$ and $\chi(\widehat{h})=h$.
Denote $\M=\text{L}(G)$, $\P=\text{L}(A)$, $\Nn=\text{L}(H)$ and $\Q=\text{L}(C^{(I)})$. Let $n$ be the smallest integer such that $n\geq t$. Denote $\sN=\Nn\overline{\otimes}\mathbb M_n(\mathbb C)$ and $\widetilde{\mathcal Q}=\Q\overline{\otimes}\mathbb D_n(\mathbb C)$.
Let $p\in\widetilde{\mathcal Q}$ be a projection such that $(\tau\otimes\text{Tr})(p)=t$ and identify $\Nn^t=p\sN p$. For $1\leq i\leq n$, let $e_i=\textbf{1}_{\{i\}}\in\mathbb D_n(\mathbb C)$.
Let $$\sR:=\theta(\P)'\cap p\sN p.$$ 
Since $D$ is an ICC subgroup of a hyperbolic group, $\theta(\M)\subset\sN_{p\sN p}(\theta(\P))''$ has the relative property (T) and $\theta(\M)\nprec_\sN\widetilde{\mathcal Q}$, Theorem \ref{relativeT} implies that 
\begin{equation}\label{theta(P)}
    \theta(\P)\prec_{\sN}^{\text{s}}\widetilde{\mathcal Q}.
\end{equation}
As $D$ is a nontrivial ICC subgroup of a hyperbolic group, it is acylindrically hyperbolic. Since $\text{Stab}_D(i)$ is amenable for every $i\in I$, \cite[Lemma 4.11(a)]{CIOS1} gives $\{i\in I\mid g\cdot i\not=i\}$ is infinite for every $g\in D\setminus\{1\}$. 
By  Corollary \ref{solidity} 
we get that $\sR$ is amenable.
Since $\theta(\M)\subset\sN_{p\sN p}(\sR)''$ has the relative property (T),  applying again Theorem \ref{relativeT} gives that
\begin{equation}\label{theta(P)}
\sR\prec_{\sN}^{\text{s}}\widetilde{\mathcal Q}.
\end{equation}

Since $\widetilde{\mathcal Q}\subset \sN$ is a Cartan subalgebra, by combining \eqref{theta(P)} with Lemma \ref{conj1}, we get that after replacing $\theta$ with $\text{Ad}(u)\circ\theta$, for some $u\in\sU(p\sN p)$, we may assume that 
\begin{equation}\label{theta(PP)}
\theta(\P)\subset\widetilde{\mathcal Q} p\subset\sR.
\end{equation}

If $g\in B$, then $\theta(u_{\widehat{g}})$ normalizes $\theta(\P)$ and thus $\sR$. Denote $\sigma_g=\text{Ad}(\theta(u_{\widehat{g}}))\in\text{Aut}(\sR)$. Then $\sigma=(\sigma_g)_{g\in B}$ defines an action of $B$ on $\sR$ which leaves $\theta(\P)$ invariant.
The restriction of $\sigma$ to $\theta(\P)$ is free since it is isomorphic to the conjugation action of $ B$ on $\text{L}(A)$, which is free by the hypothesis condition. Thus,
Lemma \ref{conj2} yields an action $\beta=(\beta_g)_{g\in B}$ of $B$ on  $\sR$ satisfying 
\begin{enumerate}
\item for every $g\in B$ we have $\beta_g=\sigma_g\circ\text{Ad}(\omega_g)=\text{Ad}(\theta(u_{\widehat{g}})\omega_g)$, for some $\omega_g\in\sU(\sR)$, and
\item $\widetilde{\mathcal Q} p$ is $\beta( B)$-invariant and the restriction of $\beta$ to $\widetilde{\mathcal Q} p$ is free.
\end{enumerate}

Our next goal is to apply Theorem \ref{SOE}.
Consider the action $\alpha=(\alpha_h)_{h\in D}$ of $D$ on $\Q$ given by $\alpha_h=\text{Ad}(v_{\widehat{h}})$ for $h\in D$.
Let $(Y,\nu)$ be the dual of $C$ with its Haar measure. Let $(X,\mu)=(Y^{I},\nu^{I})$ and $(\widetilde X,\widetilde\mu)=(X\times\mathbb Z/n\mathbb Z,\mu\times c)$, where $c$ denotes the counting measure on $\mathbb Z/n\mathbb Z$.  Identify $\Q=\text{L}^{\infty}(X)$ and $\widetilde{\mathcal Q}=\text{L}^{\infty}(\widetilde X)$.
Denote still by $\alpha$ the corresponding measure preserving action $D\curvearrowright^{\alpha} (X,\mu)$ and let  $D\times\mathbb Z/n\mathbb Z\curvearrowright^{\widetilde\alpha}(\widetilde X,\widetilde\mu)$ be the action given by $(g,a)\cdot (x,b)=(g\cdot x,a+b)$. Let $X_0\subset\widetilde X$ be a measurable set such that $p=\textbf{ 1}_{X_0}$.
Since $\widetilde{\mathcal Q} p=\text{L}^{\infty}(X_0)$ is $\beta(B)$-invariant, we get a measure preserving action $B\curvearrowright^\beta (X_0,\widetilde\mu_{|X_0})$.

Note that $\widetilde{\Q}\subset \sN$ is a Cartan subalgebra and the p.m.p. equivalence relation associated to the inclusion $\widetilde{\Q}\subset \sN$ \cite{FM77}
is equal to $\sR(D\times\mathbb Z/n\mathbb Z\curvearrowright^{\widetilde\alpha}\widetilde X)$.
Since the restriction of $\beta$ to $\widetilde{\mathcal Q} p$ is implemented by unitaries in $p\sN p$, we deduce that
\begin{equation}
\text{$\beta( B)\cdot x\subset\widetilde\alpha(D\times\mathbb Z/n\mathbb  Z)\cdot x$, for almost every $x\in X_0$.}
\end{equation}

Since $B$ has property (T) and $\alpha$ is built over $D\curvearrowright I$ as a consequence of Lemma \ref{bover}, applying Lemma \ref{wmix}, we find a partition $X_0=\sqcup_{i=1}^lX_i$ into non-null measurable sets, for some  $l\in\mathbb N\cup\{\infty\}$,  and a finite index subgroup $ S_i< B$ such that $X_i$ is $\beta( S_i)$-invariant and
 the restriction of $\beta_{| S_i}$ to $X_i$ is weakly mixing, for all $i$.

Let $1\leq i\leq l$. If $j\in I$ and $g\in D\setminus\{1\}$, then $\text{Stab}_D(j)$ is amenable and $\text{C}_{D}(g)$ is virtually cyclic. Since the action $\beta$ is free and $ S_i$ has property (T),  we can find a sequence $(h_m)\subset S_i$  such that we have $\widetilde\mu(\{x\in X_0\mid \beta_{h_m}(x)\in \widetilde\alpha(s(\text{Stab}_D(j)\times\mathbb Z/n\mathbb Z)t)(x)\})\rightarrow 0$ and
$\widetilde\mu(\{x\in X_0\mid \beta_{h_m}(x)\in \widetilde\alpha(s(\text{C}_D(g)\times\mathbb Z/n\mathbb Z)t)(x)\})\rightarrow 0$, as $m\rightarrow\infty$, for every $s,t\in D\times\mathbb Z/n\mathbb Z$.

  By Theorem \ref{SOE} we can find an injective group homomorphism $\varepsilon_i: S_i\rightarrow D$ and $\varphi_i\in [\sR(D\times\mathbb Z/n\mathbb Z\curvearrowright^{\widetilde\alpha}\widetilde X)]$ such that $\varphi_i(X_i)=X\times\{i\}\equiv X$ and $\varphi_i\circ{\beta_h}_{|X_i}=\alpha_{\varepsilon_i(h)}\circ{\varphi_i}_{|X_i}$, for all $h\in S_i$. In particular,  $\widetilde\mu(X_i)=1$. Thus,  $t=\widetilde\mu(X_0)=\sum_{i=1}^l\widetilde\mu(X_i)=l\in\mathbb N$. Since $n$ is the smallest integer with $n\geq t$, we get that
  $n=t=l$ and $p=1_{\mathscr N}$.

For $1\leq i\leq n$, let $p_i=\textbf {1}_{X_i}\in\widetilde\Q$ and $u_i\in\sN_{\sN}(\widetilde{\mathcal Q})$ such that $u_iau_i^*=a\circ\varphi_i^{-1}$, for every $a\in\widetilde{\mathcal Q}$. Then $u_ip_iu_i^*=1\otimes e_i$ and since $\beta_h=\text{Ad}(\theta(u_{\widehat{h}})\omega_h)$ we find $(\zeta_{i,h})_{h\in S_i}\subset\sU({\Q})$ with
\begin{equation}\label{th}
\text{$u_i\theta(u_{\widehat{h}})\omega_hp_iu_i^*=\zeta_{i,h}v_{\widehat{\varepsilon_i(h)}}\otimes e_i$, for every $h\in S_i$.}
\end{equation}

After replacing $ S_i$ by $ S=\cap_{i=1}^n S_i$ we may assume that $ S_i= S$, for all $1\leq i\leq n$.
We will prove the conclusion for $K=\pi^{-1}( S)$. 
Let $\M_0=\text{L}(K)$.
To prove the conclusion, it suffices to find a projection $p_0\in\theta(\M_0)'\cap\sN$ with $(\tau\otimes\text{Tr})(p_0)=k\in\{1,\ldots,n\}$, homomorphisms $\delta:K\rightarrow H$, $\rho:K\rightarrow\sU_k(\mathbb C)$, $w\in\sU(\sN)$ such that $\delta$ is injective, $wp_0w^*=1\otimes (\sum_{i=1}^ke_i)$ and $w\theta(u_g)p_0w^*=v_{\delta(g)}\otimes\rho(g)$, for all $g\in K$.
 Indeed, once we have this assertion, the conclusion will follow by a maximality argument. This concludes the first part of the proof.

The rest of the proof, which is divided between four claims (Claims \ref{P0}-\ref{Ji}), is devoted to proving the last assertion. We begin with the following:

\begin{claim}\label{P0}
There are $1\leq k\leq n$ and a homomorphism $\varepsilon: S\rightarrow D$ such that, after renumbering, we have $p_0=\sum_{i=1}^kp_i\in\theta(\M_0)'\cap \sN$ and can take $\varepsilon_i=\varepsilon$, for every $1\leq i\leq k$.
\end{claim}

\begin{proof}[Proof of Claim \ref{P0}.]
We start by showing that if $1\leq i,j\leq n$ and $p_i\sR p_j\not=\{0\}$, then $\varepsilon_i$ and $\varepsilon_j$ are conjugate. Let $x\in (\sR)_1$ with $p_ixp_j\not=0$. Then $x_h=\text{Ad}(\theta(u_{\widehat{h}})\omega_h)(p_ix)=\beta_h(p_ix)\in (\sR)_1$ and $(\theta(u_{\widehat{h}})\omega_hp_i)xp_j=x_h(\theta(u_{\widehat{h}})\omega_hp_j)$, for every $h\in S$. 
Using \eqref{th} we derive that \begin{equation}\label{zzeta}\text{$\big(u_i^*(\zeta_{i,h}v_{\widehat{\varepsilon_i(h)}}\otimes e_i)u_i \big)xp_j=x_h \big(u_j^*(\zeta_{j,h}v_{\widehat{\varepsilon_j(h)}}\otimes e_j)u_j\big)$, for every $h\in S$.}\end{equation}

For every subset $F\subset D$, let $P_{F}$ be the orthogonal projection from $\text{L}^2(\sN)$ onto the $\|\cdot\|_2$-closed linear span of $\{v_g\otimes x\mid g\in\chi^{-1}(F),x\in\mathbb M_n(\mathbb C)\}$. Since $(\zeta_{i,h})_{h\in S}, (\zeta_{j,h})_{h\in S}\subset\sU(\Q)$ and $(x_h)_{h\in S}\subset (\sR)_1$, then using basic $\|\cdot\|_2$-approximations and also
 \eqref{theta(P)} in combination with \cite[Lemma 2.5]{Va10} (for b)), we  can find finite $F\subset D$ so that for every $h\in D$ we have 
 \begin{equation*}\begin{split}& \text{a)} \qquad \|\big(u_i^*(\zeta_{i,h}v_{\widehat{\varepsilon_i(h)}}\otimes e_i)u_i \big)xp_j-P_{F\varepsilon_i(h)F}(\big(u_i^*(\zeta_{i,h}v_{\widehat{\varepsilon_i(h)}}\otimes e_i)u_i \big)xp_j)\|_2<\frac{\|p_ixp_j\|_2}{2}, \; \text{ and  }\\ &
 \text{b)} \qquad \|x_h \big(u_j^*(\zeta_{j,h}v_{\widehat{\varepsilon_j(h)}}\otimes e_j)u_j\big)-P_{F\varepsilon_j(h)F}(x_h \big(u_j^*(\zeta_{j,h}v_{\widehat{\varepsilon_j(h)}}\otimes e_j)u_j\big))\|_2<\frac{\|p_ixp_j\|_2}{2}.\end{split}\end{equation*}
 
 Combining a)-b) with \eqref{zzeta}, we derive that $F\varepsilon_i(h)F\cap F\varepsilon_j(h)F\not=\emptyset$, for every $h\in S$.
Since $\varepsilon_i$ is injective and $ S< B$ has finite index, $\varepsilon_i( S)$ is an infinite property (T) group.
For any $g\in D\setminus\{1\}$, $\text{C}_{D}(g)$ is virtually abelian and hence $\{\varepsilon_i(h)g\varepsilon_i(h)^{-1}\mid h\in S\}$ is infinite.  By \cite[Lemma 7.1]{BV13} we find $g\in D$ such that $\varepsilon_i(h)=g\varepsilon_j(h)g^{-1}$, for every $h\in S$. This proves our assertion that $\varepsilon_i$ and $\varepsilon_j$ are conjugate.

Let $\varepsilon=\varepsilon_1: S\rightarrow D$.
After renumbering, we may assume that $\varepsilon_1,\ldots,\varepsilon_k$ are conjugate to $\varepsilon$ and $\varepsilon_{k+1},\ldots,\varepsilon_n$ are not conjugate to $\varepsilon$, for some $1\leq k\leq n$. The previous paragraph implies that  $p_i\sR p_j=\{0\}$, for every $1\leq i\leq k$ and $k+1\leq j\leq n$. Thus, $p_0=\sum_{i=1}^kp_i$ belongs to the center of $\sR$. As $p_0$ commutes with $\theta(u_{\widehat{g}})\omega_g$ and $\omega_g\in\sU(\sR)$, then $p_0$ commutes with $\theta(u_{\widehat{g}})$, for every $g\in S$. 
Since $p_i\in \widetilde\Q\subset \sR=\theta(\P)'\cap\sN$, for every $1\leq i\leq n$, $p_0\in \theta(\P)'\cap\sN$. Thus, $p_0$ also commutes with $\theta(\P)$.
Since $\P$ and $(u_{\widehat{g}})_{g\in S}$ generate $\M_0$, we get that $p_0\in\theta(\M_0)'\cap \sN$. Moreover, if $1\leq i\leq k$,  there is $g_i\in D$ such that $\varepsilon_i(h)=g_i\varepsilon(h)g_i^{-1}$, for every $h\in S$. After replacing $\varphi_i$ by $\alpha_{g_i^{-1}}\circ\varphi_i$ we may assume that $\varepsilon_i=\varepsilon$, for every $1\leq i\leq k$.
 \end{proof}

Let $u=\sum_{i=1}^ku_ip_i$ and $e=\sum_{i=1}^ke_i$. Then $u$ is a partial isometry, $uu^*=1\otimes e, u^*u=p_0$ and $u\widetilde{\mathcal Q} p_0u^*=\widetilde{\mathcal Q}(1\otimes e)$. If we let $\zeta_h=\sum_{i=1}^k\zeta_{i,h}\otimes e_i\in\sU(\widetilde{\mathcal Q}(1\otimes e))$, then \eqref{th} gives that
\begin{equation}\label{thetap_0}\text{$u\theta(u_{\widehat{h}})\omega_hp_0u^*=\zeta_h(v_{\widehat{\varepsilon(h)}}\otimes e)$, for every $h\in S$.}\end{equation}

Identify $\sN_1:=(1\otimes e)\sN(1\otimes e)$ and $\Q_1:=\widetilde{\mathcal Q}(1\otimes e)$ with $\Nn\overline{\otimes}\mathbb M_k(\mathbb C)$ and $\mathcal Q\overline{\otimes}\mathbb D_k(\mathbb C)$, respectively.
Consider the unital $*$-homomorphism $\theta_1:\M_0\rightarrow \sN_1$ given by $\theta_1(x)=u\theta(x)p_0u^*$, for every $x\in \M_0$, and let $\sR_1=\theta_1(\P)'\cap \sN_1$.  Letting $w_h=u\omega_hp_0u^*\in\sU(\sR_1)$, then \eqref{thetap_0} rewrites as
\begin{equation}\label{thetap}\text{$\theta_1(u_{\widehat{h}})w_h=\zeta_h(v_{\widehat{\varepsilon(h)}}\otimes 1)$, for every $h\in S$. }\end{equation}

By \eqref{theta(PP)}, $\Q_1\subset\sR_1$.
Since $(\theta_1(u_{\widehat{h}})w_h)_{h\in B}$ normalizes $\sR_1$ and $(\zeta_h)_{h\in S}\subset\sU(\Q_1)$, \eqref{thetap} implies that $(v_{\widehat{\varepsilon(h)}}\otimes 1)_{h\in S}$ normalizes $\sR_1$.  Thus,  $\eta_h=\zeta_h\text{Ad}(v_{\widehat{\varepsilon(h)}}\otimes 1)(w_h^*)\in\sU(\sR_1)$ and
\begin{equation}\label{etah}
\text{$\theta_1(u_{\widehat{h}})=\eta_h(v_{\widehat{\varepsilon(h)}}\otimes 1)$, for every $h\in S$.}
\end{equation}

\begin{claim}\label{T} $\sR_1={\mathcal Q}\overline{\otimes}\T$, for a von Neumann subalgebra $\T\subset\mathbb M_k(\mathbb C)$.
\end{claim}
\begin{proof}[Proof of Claim \ref{T}.]
First we show that $\sR_1\subset {\mathcal Q}\overline{\otimes}\mathbb M_k(\mathbb C)$.
To this end, since $\varepsilon(S)$ is an infinite property (T) group and $\text{C}_D(g)$ is virtually abelian for every $g\in D\setminus\{1\}$, there is a sequence $(h_m)\subset S$ such that for every $g_1,g_2\in D$ and $k\in D\setminus\{1\}$ we have $g_1\text{Ad}(\varepsilon(h_m))(k)g_2\not=1$, for every $m$ large enough. 

Next we claim that $\|\text{E}_{\Q}(a_1\text{Ad}(v_{\widehat{\varepsilon(h_m)}})(b)a_2)\|_2\rightarrow 0$, for every $a_1,a_2\in\Nn$ and $b\in\Nn\ominus\Q$. We only have to check this for $a_1=v_{g_1},a_2=v_{g_2},b=v_k$, where $g_1,g_2\in H,k\in H\setminus C^{(I)}$. In this case $\chi(k)\not=1$, thus $\chi(g_1\text{Ad}(\widehat{\varepsilon(h_m)})(k)g_2))=\chi(g_1)\text{Ad}(\varepsilon(h_m))(\chi(k))\chi(g_2)\not=1$ and therefore $\text{E}_{\Q}(a_1\text{Ad}(v_{\widehat{\varepsilon(h_m)}})(b)a_2)=0$, for every $m$ large enough. 

The previous paragraph further implies that \begin{equation}\label{conv'}\|\text{E}_{\Q_1}(a_1\text{Ad}(v_{\widehat{\varepsilon(h_m)}}\otimes 1)(b)a_2)\|_2\rightarrow 0, \text{ for every } a_1,a_2\in\sN_1\text{ and } b\in\sN_1\ominus (\Q\overline{\otimes}\mathbb M_k(\mathbb C)).\end{equation}  Let $\sZ(\sR_1)$ be the center of $\sR_1$. Since $\Q_1\subset\sN_1$ is a masa we have $\sZ(\sR_1) \subseteq \Q_1\subseteq \sR_1$. Moreover, by \eqref{theta(P)} we have $\sR_1\prec_{\sN_1}^{\text{s}}\Q_1$ and since $\Q_1$ is abelian we get that $\sR_1$ is a type I von Neumann algebra. Hence there exist projections  $(r_i)\subset \sZ(\sR_1)$ with $\sum_i r_i=1$ such that $\sR_1 r_i \cong (\sZ(\sR_1) r_i) \otimes\mathbb M_{n_i}(\mathbb C)$, for some integer $n_i$. In particular, $\sZ(\sR_1) r_i \subseteq \sR_1 r_i$ admits a finite Pimsner-Popa basis \cite{PP86} and therefore so does $\Q_1 r_i \subseteq \sR_1 r_i$, for all $i$. Combining this with \eqref{conv'}, basic approximations show that 

\begin{equation}\label{ortho}\text{$\|\text{E}_{\sR_1}(\text{Ad}(v_{\widehat{\varepsilon(h_m)}}\otimes 1)(b))\|_2\rightarrow 0$, for every $b\in\sN_1\ominus (\Q\overline{\otimes}\mathbb M_k(\mathbb C))$.} \end{equation}


To see that $\sR_1\subset {\mathcal Q}\overline{\otimes}\mathbb M_k(\mathbb C)$, let $a\in \sR_1$. Put $c=\text{E}_{ {\mathcal Q}\overline{\otimes}\mathbb M_k(\mathbb C)}(a)$ and $b=a-c$.
Since $v_{\widehat{\varepsilon(h_m)}}\otimes 1$ normalizes $\sR_1$, $\text{Ad}(v_{\widehat{\varepsilon(h_m)}}\otimes 1)(a)\in\sR_1$ and thus $\|\text{E}_{\sR_1}(\text{Ad}(v_{\widehat{\varepsilon(h_m)}}\otimes 1)(a))\|_2=\|a\|_2$, for every $m$. 
On the other hand, $\|\text{E}_{\sR_1}(\text{Ad}(v_{\widehat{\varepsilon(h_m)}}\otimes 1)(b))\|_2\rightarrow 0$ by \eqref{ortho}. Thus, we get that $\|\text{E}_{\sR_1}(\text{Ad}(v_{\widehat{\varepsilon(h_m)}}\otimes 1)(c))\|_2\rightarrow \|a\|_2$. Since $\|\text{E}_{\sR_1}(\text{Ad}(v_{\widehat{\varepsilon(h_m)}}\otimes 1)(c))\|_2\leq \|c\|_2\leq \|a\|_2$, we conclude that $\|c\|_2=\|a\|_2$. Hence, $a=c\in {\mathcal Q}\overline{\otimes}\mathbb M_k(\mathbb C).$ This proves that  $\sR_1\subset {\mathcal Q}\overline{\otimes}\mathbb M_k(\mathbb C)$.

Therefore, we have that  $\Q_1={\mathcal Q}\overline{\otimes}\mathbb D_k(\mathbb C)\subseteq\sR_1\subseteq {\mathcal Q}\overline{\otimes}\mathbb M_k(\mathbb C).$
Since $\Q=\text{L}^{\infty}(X,\mu)$, we can disintegrate $\sR_1=\int_X^{\oplus}\T_x\;\text{d}\mu(x)$, where $(\T_x)_{x\in X}$ is a measurable field of von Neumann subalgebras of $\mathbb M_k(\mathbb C)$ containing $\mathbb D_k(\mathbb C)$. (Note: although we will not use this fact we note that, as pointed out to us by the referee, any intermediate von Neumann algebra $\mathbb D_k(\mathbb C)\subset\mathcal S\subset\mathbb M_k(\mathbb C)$ is of the form $\mathcal S=\oplus_jp_j\mathbb M_k(\mathbb C)p_j$, for a partition of unity $(p_j)$ with projections from $\mathbb D_k(\mathbb C)$).
We denote $$\text{$\gamma_h=\text{Ad}(v_{\widehat{\varepsilon(h)}})\in\text{Aut}(\Q)$, for $h\in S$.}$$ Then $\gamma=(\gamma_h)_{h\in S}$ defines an action of $ S$ on $\Q$.  Since $(v_{\widehat{\varepsilon(h)}}\otimes 1)_{h\in S}$ normalizes $\sR_1$, we have $\T_{\gamma_h(x)}=\T_x$, for every $h\in S$ and almost every $x\in X$.  Lemma \ref{bover} implies that $\gamma$ is built over the action $S\curvearrowright I$ given by $h\cdot i=\varepsilon(h)i$.
Since $\text{Stab}_D(i)$ is amenable, $\varepsilon$ is injective and $S$ is an infinite group with property (T), the action $S\curvearrowright I$ has infinite orbits. This implies that $\gamma$ is weakly mixing, hence ergodic. Thus, one can find a von Neumann subalgebra $\T\subset\mathbb M_k(\mathbb C)$ such that $\T_x=\T$, for almost every $x\in X$, which proves the claim.  \end{proof}

Next, let $\sU=\sU(\T)/\sU(\sZ(\T))$, where $\sZ(\T)$ is the center of $\T$, and $q:\sU(\T)\rightarrow\sU$ be the quotient homomorphism. We continue with the following claim:

\begin{claim}\label{nu}
There are maps $\xi: S\rightarrow\sU(\T)$ and $\nu: S\rightarrow\sU(\sZ(\sR_1))$ such that the map $ S\ni h\mapsto q(\xi_h)\in\sU$ is a homomorphism and after replacing  $\theta_1$ by $\emph{Ad}(u)\circ\theta_1$, for some $u\in\sU(\sR_1)$, we have that
$\theta_1(u_{\widehat{h}})=\nu_h(v_{\widehat{\varepsilon(h)}}\otimes \xi_h)$, for every $h\in S$.
\end{claim}

\begin{proof}[Proof of Claim \ref{nu}]
Note that $(\text{Ad}(\theta_1(u_{\widehat{h}})))_{h\in S}$ and $(\gamma_h\otimes\text{Id}_{\T})_{h\in S}$ define actions of $ S$ on $\sR_1=\Q\overline{\otimes} \T$. Combining this observation with \eqref{etah} gives that
\begin{equation}\label{cocycle}
\text{$\eta_{h_1h_2}^*\eta_{h_1}(\gamma_{h_1}\otimes\text{Id})(\eta_{h_2})\in\sZ(\sR_1)=\Q\overline{\otimes}\sZ(\T)$, for every $h_1,h_2\in S$.}
\end{equation}

Viewing every $\eta\in\sU(\sR_1)$ as a measurable function $\eta:X\rightarrow\sU(\T)$, \eqref{cocycle} rewrites as $\eta_{h_1h_2}(x)^*\eta_{h_1}(x)\eta_{h_2}(\gamma_{h_1}^{-1}x)\in\sU(\sZ(\T))$, for every $h_1,h_2\in S$ and almost every $x\in X$. 
 Then $c : S\times X\rightarrow\sU$ given by $c(h,x)=q(\eta_h(x))$ is a 1-cocycle for $\gamma$.
Since $\sU(\sZ(\T))$ is a closed central subgroup of the compact Polish group $\sU(\T)$, $\sU$ is a compact Polish group with respect to the quotient topology.
In particular, $\sU$ is a $\sU_{\text{fin}}$ group (see \cite[Lemma 2.7]{Po05}). Since $S$ has property (T), $\gamma$ is built over $S\curvearrowright I$ and $S\curvearrowright I$ has infinite orbits, Theorem \ref{CS} implies that
$c$ is cohomologous to a homomorphism $\psi: S\rightarrow\sU$. 

Let $u:X\rightarrow\sU(\T)$ be a measurable map satisfying $q(u(x))c(h,x)q(u(\gamma_{h^{-1}}(x)))^{-1}=\psi_h$, for every $h\in S$ and almost every $x\in X$.
 Let $\xi: S\rightarrow\sU(\T)$ be  such that $q(\xi_h)=\psi_h$. Thus, we find a measurable map $\nu_h:X\rightarrow\sU(\sZ(\T))$ such that $u(x)\eta_h(x)u(\gamma_{h^{-1}}(x))^{-1}=\nu_h(x)\xi_h$, for every $h\in S$ and almost every $x\in X$.
Equivalently, $u\in\sU(\sR_1)$ and $\nu: S\rightarrow\sU(\sZ(\sR_1))$ satisfy $u\eta_h(\gamma_h\otimes\text{Id})(u)^*=\nu_h(1\otimes\xi_h)$, for every $h\in S$. This implies the claim.
 \end{proof}

We are now ready to finish the proof of the main assertion.
Let $g\in K=\pi^{-1}( S)$  and put $h=\pi(g)\in S$. Then $a=g\widehat{h}^{-1}\in A$ and $\theta_1(u_a)\in \sZ(\sR_1)$. Denoting $k_g=\theta_1(u_a)\nu_h$ and using Claim \ref{nu}, we get that  $k_g\in\sU(\sZ(\sR_1))$ and
\begin{equation}\label{w_g}
\text{$\theta_1(u_g)=k_g(v_{\widehat{\varepsilon(\pi(g))}}\otimes \xi_{\pi(g)})$, for every $g\in K$.}
\end{equation}

Our final claim is the following.

\begin{claim}\label{Ji}
There are maps $c:K\rightarrow C^{(I)}$, $e:K\rightarrow\sZ(\T)$ and $z\in \sU(\sZ(\sR_1))$ such that
$k_g=(v_{c_g}\otimes e_g)z^*(\gamma_{\pi(g)}\otimes\emph{Id})(z)$, for every $g\in K$.
\end{claim}

\begin{proof}[Proof of Claim \ref{Ji}.]
We define $s:K\times K\rightarrow C^{(I)}$  by $s_{g,h}=\widehat{\varepsilon(\pi(g))}\widehat{\varepsilon(\pi(h))}\widehat{\varepsilon(\pi(gh))}^*$ and  $d:K\times K\rightarrow\sU(\sZ(\T))$ by $d_{g,h}=\xi_{\pi(gh)}\xi_{\pi(h)}^*\xi_{\pi(g)}^*$.
Since $\theta_1(u_{g})\theta_1(u_{h})=\theta_1(u_{gh})$, \eqref{w_g} gives
\begin{equation}\label{2c}
\text{$v_{s_{g,h}}\otimes 1=\big(k_{gh}k_g^*(\gamma_{\pi(g)}\otimes\text{Id})(k_h)^*\big)\big(1\otimes d_{g,h}\big)$, for every $g,h\in K$.}
\end{equation}

Note that the map $K\times K\ni(g,h)\mapsto v_{s_{g,h}}\in \sU(\Q)$ is a $2$-cocycle for the action $K\curvearrowright^{\gamma\circ\pi}\Q$ and the map $K\times K\ni (g,h)\mapsto d_{g,h}\in\sU(\sZ(\T))$ is a $2$-cocycle for the trivial action. Since $\gamma$ is built over $S\curvearrowright I$, $\gamma\circ\pi$ is built over the action $K\curvearrowright I$ given by $k\cdot i=\pi(k)\cdot i$, for every $k\in K$ and $i\in I$. Since the action $K\curvearrowright I$ has infinite orbits, $\gamma\circ\pi$ is weakly mixing.
If $n\in\mathbb N$, then the action $(\gamma\circ\pi)^{\otimes n}$ is built over the diagonal product action $K\curvearrowright I^n$, which has infinite orbits.
 Since $K$ has property (T), Theorem \ref{CS} implies that $(\gamma\circ\pi)^{\otimes n}$ is $\sU_{\text{fin}}$-cocycle superrigid.
Altogether, we deduce that $\gamma\circ\pi$ satisfies the hypothesis of Theorem \ref{jiang}.

Let $\{p_1,\ldots,p_l\}$ be the minimal projections of $\sZ(\T)$. Let $1\leq i\leq l$.
Since $d_{g,h}p_i\in\mathbb Tp_i$, for every $g,h\in K$, by using \eqref{2c} and applying Theorem \ref{jiang} there are maps $c:K\rightarrow C^{(I)}$ and $e^i:K\rightarrow\mathbb T$ such that $v_{s_{g,h}}=v_{c_{gh}}v_{c_g}^*\gamma_{\pi(g)}(v_{c_h})^*$ and $d_{g,h}p_i=\overline{e_{gh}^i}{e_g^i}{e_h^i}p_i$, for all $g,h\in K$.
Define $e:K\rightarrow\sU(\sZ(\T))$ by letting $e_g=\sum_{i=1}^le_g^ip_i$, for every $g\in K$. Then $d_{g,h}=e_{gh}^*e_ge_h$, for all $g,h\in K$.
Using \eqref{2c}, we get that the map $K\ni g\mapsto (v_{c_g}^*\otimes e_g^*)k_g\in\sU(\sZ(\sR_1))$ is a $1$-cocycle for the action $K\curvearrowright^{\gamma\circ\pi\otimes\text{Id}}\sZ(\sR_1)=\Q\overline{\otimes}\sZ(\T)$. Since $\gamma\circ\pi$ is $\sU_{\text{fin}}$-cocycle superrigid, we deduce the existence of $z\in  \sU(\sZ(\sR_1))$ such that the claim holds.
\end{proof}

Finally, \eqref{w_g} and Claim \ref{Ji} imply that $\theta_1(u_g)=z^*(v_{c_g\widehat{\varepsilon(\pi(g))}}\otimes e_g\xi_{\pi(g)})z$, for every
 $g\in K$. Then $\delta:K\rightarrow H$ and $\rho:K\rightarrow\sU(\T)\subset\sU_k(\mathbb C)$ given by $\delta(g)=c_g\widehat{\varepsilon(\pi(g))}$ and $\rho(g)=e_g\xi_{\pi(g)}$ must be homomorphisms. 
Then $\theta_1(u_g)=z^*(v_{\delta(g)}\otimes\rho(g))z$, for every $g\in K$. By construction $A\subset K$ and $\delta(A)\subset C^{(I)}$.
 Moreover, if $g\in\ker(\delta)$, then we have that $\theta_1(u_g)=z^*(1\otimes\rho(g))z\in z^*(1\otimes\mathbb M_k(\mathbb C))z$. This implies that $\ker(\delta)$ must be finite. Since $G$ and thus $K$ are icc it follows that $\delta$ is injective. This finishes the proof of the main assertion.

To prove the moreover assertion, assume that $t=1$. Then $\theta:\M\rightarrow \Nn$ is a unital $*$-homomorphism and after replacing it by $\text{Ad}(u)\circ\theta$, for some $u\in\sU(\Nn)$, we have that $\theta(u_k)=\rho(k)v_{\delta(k)}$, for every $k\in K$,  for some homomorphisms $\rho:K\rightarrow\mathbb T$ and $\delta:K\rightarrow H$. Moreover, we may assume that $K$ is normal in $G$. If $g\in G$, then \begin{equation}\label{thetau_g}\text{$\rho(k)\theta(u_g)v_{\delta(k)}\theta(u_g)^*=\theta(u_{gkg^{-1}})=\rho(gkg^{-1})v_{\delta(gkg^{-1})}$, for every $k\in K$.}\end{equation}
 
 Write $\theta(u_g)^*=\sum_{h\in H}c_hv_h$. If $\epsilon>0$ is such that $F=\{h\in H\mid |c_h|>\epsilon\}$ is nonempty, then \eqref{thetau_g} implies that $\delta(k)F\delta(gkg^{-1})^{-1}=F$ and thus $\delta(k)FF^{-1}\delta(k)^{-1}=FF^{-1}$, for every $k\in K$. However, since $\delta(K)$ is an infinite property (T) group and $\text{C}_H(h)$ is virtually cyclic, the set $\{\delta(k)h\delta(k)^{-1}\mid k\in K\}$ is infinite, for every $h\in H\setminus\{1\}$. Thus, $FF^{-1}=\{1\}$, so $F$ consists of a single element. Since this holds for all $\epsilon>0$, we get that $\theta(u_g)\in\mathbb T(v_h)_{h\in H}$. Thus, $\delta$ and $\rho$ extend to homomorphisms $\rho:G\rightarrow\mathbb T$ and $\delta:G\rightarrow H$ such that $\theta(u_g)=\rho(g)v_{\delta(g)}$, for every $g\in G$. Since $\delta(A)\subset C^{(I)}$, the moreover assertion follows.
\end{proof}



 Theorem \ref{rigembed} leads to a complete description of all $*$-homomorphisms $\theta:\text{L}(G)\rightarrow\text{L}(H)^t$. To explain this, we assume the setting of Theorem \ref{rigembed} and introduce some terminology from \cite[Section 2]{PV21}. 
 
 Let $K<G$ be a finite index subgroup. 
  If $\varepsilon:K\rightarrow H$  and $\mu:K\rightarrow\mathscr U_s(\mathbb C)$ are homomorphisms, for some $s\in\mathbb N$, we denote by $\pi_{\varepsilon,\mu}:K\rightarrow \sU(\text{L}(H)\overline{\otimes}\mathbb M_s(\mathbb C))$ the homomorphism given by $\pi_{\varepsilon,\mu}(g)=v_{\varepsilon(g)}\otimes\mu(g)$ for  $g\in K$. If $\pi:K\rightarrow\sU(\mathcal S)$ is a homomorphism, where $\mathcal S$ is a tracial von Neumann algebra, we denote by
 $\text{Ind}_K^G(\pi):G\rightarrow\sU(\mathcal S\overline{\otimes}\mathbb M_{[G:K]}(\mathbb C))$ the induced homomorphism. Specifically, let $\varphi:G/K\rightarrow G$ be a map such that $\varphi(gK)\in gK$, for every $g\in G$, and define $c:G\times G/K\rightarrow K$ by letting $c(g,hK)=\varphi(ghK)^{-1}g\varphi(hK)\in K$, for every $g,h\in G$.
Identifying $\mathbb M_{[G:K]}(\mathbb C)=\mathbb B(\ell^2(G/K))$, we define $\text{Ind}_K^G(\pi)(g)(\xi\otimes\delta_{hK})=\pi(c(g,hK))\xi\otimes\delta_{ghK}$.
By \cite[Defintion 2.1]{PV21}, a homomorphism $G\rightarrow \sU(\text{L}(H)^t)$ is called standard if it is unitarily conjugate to a direct sum of homomorphisms of the form $\text{Ind}_K^G(\pi_{\varepsilon,\mu})$ induced from finite index subgroups $K<G$.
Theorem \ref{rigembed} implies that the restriction of any $*$-homomorphism $\theta:\text{L}(G)\rightarrow\text{L}(H)^t$  to $G=\{u_g\}_{g\in G}$ is standard. More precisely, we have:

%

\begin{thm}\label{explicit}
Let $\theta:\emph{L}(G)\rightarrow\emph{L}(H)^t$ be a unital $*$-homomorphism, for some $t>0$. Then $t\in\mathbb N$ and we can find $q\in\mathbb N$, a finite index subgroup $K_i<G$, an injective homomorphism $\varepsilon_i:K_i\rightarrow H$ and a unitary representation $\mu_i:K_i\rightarrow\mathscr U_{s_i}(\mathbb C)$, for some $s_i\in\mathbb N$,  for every $1\leq i\leq q$,  and a unitary $u\in\emph{L}(H)^t=\emph{L}(H)\overline{\otimes}\mathbb M_t(\mathbb C)$  such that $\sum_{i=1}^q[G:K_i]s_i=t$ and
$$\text{$u\theta(u_g)u^*=\emph{diag}(\emph{Ind}_{K_1}^G(\pi_{\varepsilon_1,\mu_1})(g),\ldots,\emph{Ind}_{K_q}^G(\pi_{\varepsilon_q,\mu_q})(g))$, for every $g\in G$.}$$
\end{thm}

\begin{proof}
We keep the notation from the proof of Theorem \ref{rigembed}. By Theorem \ref{rigembed}, $t\in\mathbb N$ and there are $t_1,\ldots, t_p\in\mathbb N$ with $t_1+\cdots+t_p=t$, for some $p\in\mathbb N$, a finite index subgroup $K<G$, an injective homomorphism $\delta_i:K\rightarrow H$ and a unitary representation $\rho_i:K\rightarrow\sU_{t_i}(\mathbb C)$, for every $1\leq i\leq p$, and a unitary $w\in \text{L}(H)^t=\text{L}(H)\overline{\otimes}\mathbb M_t(\mathbb C)$ such that 
$$\text{$ w\theta(u_g)w^*=\text{diag}(v_{\delta_1(g)}\otimes\rho_1(g),
\ldots, v_{\delta_p(g)}\otimes\rho_p(g))
$, for every $g\in K$.}$$

We may assume that $K<G$ is a normal subgroup.
By decomposing each $\rho_i, 1\leq i\leq q$, into a direct sum of irreducible representations, we may assume that $\rho_i$ is irreducible, for every $1\leq i\leq p$.
Further, we can find $s_1,\ldots,s_r,d_1,\ldots,d_r\in\mathbb N$, for some $r\in\mathbb N$, with $s_1d_1+\cdots+s_rd_r=t$, an injective homomorphism $\gamma_i:K\rightarrow H$ and an irreducible representation $\sigma_i:K\rightarrow\sU_{s_i}(\mathbb C)$, for every $1\leq i\leq r$, and a unitary $z\in \text{L}(H)^t$ such that for every $1\leq i\leq r$,
$\gamma_i$ is not conjugate to $\gamma_j$ or $\sigma_i$ is not unitarily conjugate to $\sigma_j$,
and after replacing $\theta$ by $\text{Ad}(z)\circ\theta$ we have
\begin{equation}\label{thetaonK}
\text{$\theta(u_g)=\text{diag}(v_{\gamma_1(g)}\otimes\sigma_1(g)\otimes I_{d_1},
\ldots, v_{\gamma_r(g)}\otimes\sigma_r(g)\otimes I_{d_r})
$, for every $g\in K$.}
\end{equation}

Here, we call homomorphisms $\gamma,\gamma':K\rightarrow H$ conjugate if $\gamma(g)=h\gamma'(g)h^{-1}$, for every $g\in K$, for some $h\in H$. We denote by $I_d\in\mathbb M_d(\mathbb C)$ the identity matrix, for every $d\in\mathbb N$, and consider the natural unital embedding $\bigoplus_{i=1}^r\big(\text{L}(H)\otimes\mathbb M_{s_i}(\mathbb C)\otimes\mathbb M_{d_i}(\mathbb C))\subset \text{L}(H)\otimes\mathbb M_t(\mathbb C).$ 

We continue with the following claim:

\begin{claim}\label{intertwiners}
Let $\gamma,\gamma':K\rightarrow H$ be injective homomorphisms
and $\sigma:K\rightarrow\sU_{s}(\mathbb C), \sigma':K\rightarrow\sU_{s'}(\mathbb C)$ be irreducible representations, for some $s,s'\in\mathbb N$.
 Assume that there exists a nonzero $x\in \emph{L}(H)\otimes\mathbb M_{s,s'}(\mathbb C)$ such that   $(v_{\gamma(g)}\otimes\sigma(g))x=x(v_{\gamma'(g)}\otimes\sigma'(g))$, for every $g\in K$.

 Then $s=s'$ and there exist $h\in H$ and $U\in\sU_s(\mathbb C)$ such that for every $g\in K$ we have that $\gamma(g)h=h\gamma'(g)$ and
$\sigma(g)U=U\sigma'(g)$. 
Moreover, if $y\in \emph{L}(H)\otimes\mathbb M_{s}(\mathbb C)$ satisfies that $(v_{\gamma(g)}\otimes\sigma(g))y=y(v_{\gamma'(g)}\otimes\sigma'(g))$, for every $g\in K$, then   $y=\alpha(v_h\otimes U)$, for some $\alpha\in\mathbb C$.

\end{claim}

\begin{proof}
Write $x=\sum_{h\in H}v_h\otimes c_h$, where $c_h\in\mathbb M_{s,s'}(\mathbb C)$ and $\sum_{h\in H}\|c_h\|_2^2=\|x\|_2^2<\infty$. Then 
\begin{equation}\label{equivariance}
    \text{$c_{\gamma(g)h\gamma'(g)^{-1}}=\sigma(g)c_h\sigma'(g)^*$, for every $g\in K, h\in G.$}
\end{equation}
  Let $\varepsilon>0$ such that $F=\{h\in H\mid \|c_h\|_2>\varepsilon\}$ is nonempty. Then \eqref{equivariance} implies that $\gamma(g)F\gamma'(g)^{-1}=F$, for every $g\in K$.   Thus, $\gamma(g)FF^{-1}\gamma(g)^{-1}=FF^{-1}$, for every $g\in K$. The same argument as in the proof of the moreover assertion of Theorem \ref{rigembed} implies that $FF^{-1}=\{1\}$, so $F$ consists of a single element. Since this holds for every small enough $\varepsilon>0$, we conclude that $\{h\in H\mid c_h\not=0\}$ has a single element. Thus, $x=v_h\otimes c$, for some $h\in H$ and nonzero $c=c_h\in\mathbb M_{s,s'}(\mathbb C)$. Hence, we have that $v_{\gamma(g)h}\otimes\sigma(g)c=v_{h\gamma'(g)}\otimes c\sigma'(g)$, which implies that $\gamma(g)h=h\gamma'(g)$ and $\sigma(g)c=c\sigma'(g)$, for every $g\in K$.
  Since $\sigma,\sigma'$ are irreducible, the latter relation implies that $s=s'$ and $c=\alpha U$, for some nonzero $\alpha\in \mathbb C$ and $U\in\sU_s(\mathbb C)$. This proves the main assertion.

  Let now $y\in \text{L}(H)\otimes\mathbb M_{s}(\mathbb C)$ such that $(v_{\gamma(g)}\otimes\sigma(g))y=y(v_{\gamma'(g)}\otimes\sigma'(g))$, for every $g\in K$. Note that $(v_{\gamma(g)}\otimes\sigma(g))(v_h\otimes U)=(v_h\otimes U)(v_{\gamma'(g)}\otimes\sigma'(g))$, for every $g\in K$. Thus, denoting $z=y(v_h\otimes U)^*$, we get that $(v_{\gamma(g)}\otimes\sigma(g))z=z(v_{\gamma(g)}\otimes\sigma(g))$, for every $g\in K$. By the first part of the proof we get that $z=\beta(v_k\otimes V)$, for some nonzero $\beta\in \mathbb C$, $k\in H$ and $V\in \sU_s(\mathbb C)$ such that $\gamma(g)k=k\gamma(g)$ and $\sigma(g)V=V\sigma(g)$, for every $g\in K$. As $\gamma:K\rightarrow H$ is injective and $\text{C}_H(h)$ is virtually cyclic, for every $h\in H\setminus\{1\}$, this forces $k=1$. Since $\sigma$ is irreducible, we also get that $V\in\mathbb T1$. Altogether, we get that $z\in\mathbb C1$, which proves the moreover assertion.
\end{proof}

Denote $I=\{1,\ldots,r\}$ and
for every $i\in I$, let $f_i=1\otimes I_{d_i}\otimes I_{s_i}$.
Combining \eqref{thetaonK} with Claim \ref{intertwiners} we get that $\theta(\text{L}(K))'\cap \text{L}(H)^t=\bigoplus_{i=1}^r(1\otimes I_{s_i}\otimes\mathbb M_{d_i}(\mathbb C))$. 
Since $K<G$ is normal, $\theta(u_g)$ normalizes $\theta(\text{L}(K))$ and thus $\sZ(\theta(\text{L}(K))'\cap \text{L}(H)^t)=\bigoplus_{i=1}^r\mathbb Cf_i$, for every $g\in G$. 
Hence, there exists an action $G\curvearrowright I$ such that $\theta(u_g)f_i\theta(u_g)^*=f_{g\cdot i}$, for every $g\in G$ and $i\in I$. Let $J\subset I$ be a set which intersects every $G$-orbit exactly once. 

Next, fix $i\in J$ and denote $K_i=\{g\in G\mid g\cdot i=i\}$. 
Then \eqref{thetaonK} implies that $K<K_i$.
Let $h\in K_i$. If $g\in K$, then since $K<G$ is normal, $hgh^{-1}\in K,$ and \eqref{thetaonK} gives that \begin{equation}\label{conjug}\text{$\theta(u_{hgh^{-1}})=\text{diag}(v_{\gamma_1(hgh^{-1})}\otimes\sigma_1(hgh^{-1})\otimes I_{d_1},\ldots, v_{\gamma_r(hgh^{-1})}\otimes\sigma_r(hgh^{-1})\otimes I_{d_r})$. }
\end{equation}

Since $\theta(u_{hgh^{-1}})\theta(u_h)=\theta(u_h)\theta(u_g)$, we get that $\theta(u_{hgh^{-1}})(\theta(u_h)f_i)=(\theta(u_h)f_i)\theta(u_g)$.
By combining \eqref{thetaonK} and \eqref{conjug} we conclude that for every $g\in K$ we have
\begin{equation}
\text{$(v_{\gamma_i(hgh^{-1})}\otimes\sigma_i(hgh^{-1})\otimes I_{d_i})(\theta(u_h)f_i)=(\theta(u_h)f_i)(v_{\gamma_i(g)}\otimes\sigma_i(g)\otimes I_{d_i})$.}
\end{equation}
By applying the moreover part of Claim \ref{intertwiners}, we get that $\theta(u_h)f_i=v_{\varepsilon_i(h)}\otimes \mu_i(h)$, for some $\varepsilon_i(h)\in H$ and $\mu_i(h)\in\sU_{s_id_i}(\mathbb C)$.
Then $\varepsilon_i:K_i\rightarrow H$ and $\mu_i:K_i\rightarrow\sU_{s_id_i}(\mathbb C)$ must be homomorphisms such that $\varepsilon_i(g)=\gamma_i(g)$ and $\mu_i(g)=\sigma_i(g)\otimes I_{d_i}$, for every $g\in K$. Thus, in the notation introduced before this proof, we have that $\theta(u_h)f_i=\pi_{\varepsilon_i,\mu_i}(h)$, for every $h\in K_i$.

Let $e_i=\sum_{g\in G}f_{g\cdot i}\in \sZ(\theta(\text{L}(K))'\cap \text{L}(H)^t)$. Then $e_i\in\theta(\text{L}(G))'\cap\text{L}(H)^t$, hence $e_i\in \sZ(\theta(\text{L}(G))'\cap \text{L}(H)^t)$. Since $e_i=\sum_{g\in G/K_i}\theta(u_g)f_i\theta(u_g)^*$ and the projections $\{\theta(u_g)f_i\theta(u_g)^*\mid g\in G/K_i\}$ are pairwise orthogonal, the homomorphism  $G\ni g\mapsto \theta(u_g)e_i\in\sU(e_i\text{L}(H)^te_i)$ is unitarily conjugate to the induced homomorphism $\text{Ind}_{K_i}^G(\pi_{\varepsilon_i,\mu_i})$. 

Since $\theta(u_g)=\sum_{i\in J}\theta(u_g)e_i$, for every $g\in G$, the conclusion follows.
\end{proof}

\section{Computations of invariants of II$_1$ factors}

The goal of this section is to prove Corollary \ref{continuum} and Theorem \ref{endo}.

\begin{proof}[Proof of Corollary \ref{continuum}]
By \cite[Corollary 1.7]{CIOS1} and its proof, for every infinite set of primes $\mathcal P$ we have a property (T) group $V_{\mathcal P}\in\W\R(A_{\mathcal P},B)$, where $A_\mathcal P=\bigoplus_{p\in\mathcal P}\mathbb Z/p\mathbb Z$ and $B$ is a non-trivial, ICC subgroup of a hyperbolic group. The group $B$ is built in \cite[Proposition 2.11]{CIOS1} and arises as a subgroup of $G/\ll g^k \rr$, where $G$ is a torsion-free hyperbolic group, $g\in G\setminus\{1\}$ and $k\in\mathbb N$ sufficiently large. By Corollary \ref{Cor:VCCquot}, the group $G/\ll g^k \rr$ and thus $B$ is VVC. This implies that $V_\mathcal P\in\mathcal C_0$.

Let $I$ be a one-parameter family of infinite sets of primes such that $\mathcal P\not\subset\mathcal P'$, for all $\mathcal P,\mathcal P'\in I$ with $\mathcal P\not=\mathcal P'$.
We claim that the continuum of groups $(V_\mathcal P)_{\mathcal P\in I}$ satisfies the conclusion.
Since $V_\mathcal P$ is W$^*$-superrigid by \cite[Theorem 1.3]{CIOS1}, for every $\mathcal P$, it remains to prove that $\text{L}(V_\mathcal P)\not\hookrightarrow_{\text{s}}\text{L}(V_{\mathcal P'})$, for all $\mathcal P,\mathcal P'\in I$ with $\mathcal P\not=\mathcal P'$. Otherwise, since $V_\mathcal P,V_{\mathcal P'}\in\mathcal C_0$, Theorem \ref{rigembed} would imply the existence of a finite index subgroup $W\subset V_{\mathcal P}$  and an injective homomorphism $A_{\mathcal P}\subset W$ and $\delta:W\rightarrow V_{\mathcal P'}$ such that $\delta(A_{\mathcal P})\subset A_{\mathcal P'}$. Since $A_\mathcal P$ contains an element of prime order $p$ if and only if $p\in \mathcal P$, we get that $\mathcal P\subset\mathcal P'$, which is a contradiction.
\end{proof}

\begin{proof}[Proof of Theorem \ref{endo}] Lemma \ref{trivialout} gives a non-elementary VCC hyperbolic group $H$ with property (T), trivial abelianization, and $\text{Out}(H)=1$. Thus, by Theorem \ref{Thm:Group} we can find a quotient $G$ of $H$ such that $G\in \WR(\ZZ,B\ca I )$, where $B\ca I$ is a transitive action with finite stabilizers and $B$ is a non-elementary VCC hyperbolic group whose injective endomorphisms are inner. Moreover, $G$ has property (T) and trivial abelianization since $H$ has these properties.

Let $\M=\text{L}(G)$. By Theorem \ref{rigembed}, if $\theta:\M\rightarrow \M^t$ is a unital $*$-homomorphism, for some $t>0$, then $t\in\mathbb N$. Thus, $\mathcal F(\M)=\{1\}$ and $\mathcal F_s(\M)=\mathbb N$.

 Let $\theta:\M\rightarrow\M$ be a unital $*$-homomorphism. Since $G$ has no non-trivial characters, the moreover part of  Theorem \ref{rigembed} gives an injective homomorphism $\delta:G\rightarrow G$  and a unitary $w\in\M$ such that $\delta(\mathbb Z^{(I)})\subset\mathbb Z^{(I)}$ and $\theta(u_g)=wu_{\delta(g)}w^*$, for every $g\in G$. Let $\beta:B\rightarrow B$ be the homomorphism given by $\beta(g\mathbb Z^{(I)})=\delta(g)\mathbb Z^{(I)}$, for every $g\in B\equiv G/\mathbb Z^{(I)}$. Then $\ker(\beta)=\{g\mathbb Z^{(I)}\mid \delta(g)\in\mathbb Z^{(I)}\}=\delta^{-1}(\mathbb Z^{(I)})/\mathbb Z^{(I)}$. Since $\delta$ is injective, $\delta^{-1}(\mathbb Z^{(I)})$ is abelian and thus $\ker(\beta)$ is abelian. Since $B$ is an ICC hyperbolic group, it has no non-trivial abelian normal subgroups. Therefore, $\beta$ is injective so it must be an inner automorphism of $B$. Applying  Corollary \ref{injectiveOut} implies that $\delta$ is an inner automorphism of $G$ and so $\theta$ is an inner automorphism of $\M$. This proves that $\text{End}(\M)=\text{Inn}(\M)$ and finishes the proof.
\end{proof}

\paragraph{Conflicts of interest statement} Conflicts of interest: none. 

\paragraph{Data Availability} There are no external data associated to this paper.

\end{document}